\pdfoutput=1

\documentclass[12pt,reqno]{amsart}
\usepackage[utf8]{inputenc}
\usepackage{hyperref}
\usepackage[T1]{fontenc}

\usepackage{amssymb}
\usepackage{amsmath}
\usepackage{xcolor}
\usepackage{graphicx}
\usepackage{float}

\usepackage[top=1in, bottom=1in, left=1in, right=1in, inner=2.5cm,outer=2.5cm]{geometry} 

\newcommand{\be}{\begin{equation}}
\newcommand{\ee}{\end{equation}}

\newcommand{\G}{\mathbb{G}}
\newcommand{\N}{\mathbb{N}}
\newcommand{\Z}{\mathbb{Z}}

\newcommand{\C}{\mathbb{C}}
\newcommand{\R}{\mathbb{R}}

\newcommand{\s}{s}

\newtheorem{theorem}{Theorem} [section]
\newtheorem{corollary}[theorem]{Corollary}
\newtheorem{proposition}[theorem]{Proposition}
\newtheorem{remark}[theorem]{Remark}
\newtheorem{lemma}[theorem]{Lemma}
\newtheorem{definition}[theorem]{Definition}

\usepackage[
backend=biber,
style=alphabetic,
]{biblatex}

\renewcommand{\Re}{\mathrm{Re}}

\usepackage{subfiles}

\begin{document}

\title[Characterizations of Sobolev spaces] {Mean Value Inequalities and Characterizations of Sobolev spaces on graded groups}

\author {Pablo De N\'apoli}
\address{IMAS (UBA-CONICET) and Departamento de Matem\'atica\\
	Facultad de Ciencias Exactas y Naturales - Universidad de Buenos Aires\\
	Ciudad Universitaria - 1428 Buenos Aires, Argentina.}
\email{pdenapo@dm.uba.ar}

\author{Rocío Díaz Martín}
\address{Department of Mathematics, Vanderbilt University, Nashville, Tennessee, USA. }
\email{rocio.p.diaz.martin@vanderbilt.edu}
\thanks{The first author was supported by ANPCyT under grant PICT-2018-03017, and by Universidad de
	Buenos Aires under grant 20020160100002BA. He is a member of
	CONICET, Argentina. The second author was primarily supported by a CONICET postdoctoral fellowship. 
	This project has mainly been carried out while she was a postdoctoral researcher at IAM-CONICET, a Graduate Teaching Assistant at Facultad de Ciencias Exactas y Naturales – Universidad de Buenos Aires, and an Assistant Professor at Facultad de Matem\'atica, Astronom\'ia, F\'isica y Computaci\'on (FaMAF)– Universidad Nacional de C\'ordoba, Argentina. She finished this article as a Postdoctoral Research Scholar at Vanderbilt University. The authors thank those institutions for their support.}



\begin{abstract}
    We develop characterizations for Sobolev spaces of potential type on graded Lie groups, by means of Littlewood-Paley square functions, and  Strichartz functionals involving second-order differences. A key role is played by some mean value inequalities that may be of independent interest. Indeed, the point-wise first-order inequality is generalized in two directions: on $L^p$-spaces and by considering second-order differences.
\end{abstract}

\maketitle

\section{Introduction}

Functional spaces on Lie groups have been intensively studied in the literature, since the
pioneering works of G. Folland \cite{Folland} and his joint work with E. M. Stein \cite{Folland-Stein}.
In particular, Folland introduced in \cite{Folland} Sobolev spaces of potential type in the setting
of stratified (nilpotent) Lie groups. Recently, V. Fischer and M.
Ruzhansky in \cite{FR-Sobolev} generalized the definition
to the setting of nilpotent graded Lie groups.

Let $\G$ be a homogeneous Lie group\footnote{We refer the reader to Section \ref{section-pelim} for a survey of the basic 
definitions and notations that we use.} and let $\mathfrak{g}$ be its Lie algebra. 
There exist various ways of defining Sobolev spaces on $\G$, which naturally extend the corresponding definitions in the Euclidean case. Let $1 \leq p<\infty$, 
on the one hand, 
one can consider Sobolev spaces defined by means of the vector fields of a Jacobian basis $\{X_1,\dots,X_n\}$ of $\mathfrak{g}$, that is, for $k \in \N$, the Banach space
$$ W^{k,p}(\G) = \left \{ f:\G\to\C \mid \,  X_1^{\alpha_1}\dots X_n^{\alpha_n} f \in L^p(\G) \, \; \forall \alpha \in\N_0^n, | \alpha|_{\G}  \leq k \right\}, $$
with the norm
$$ \| f \|_{W^{k,p}(\G)} = \left( \sum_{|\alpha|_{\G} \leq k}  \| X_1^{\alpha_1}\dots X_n^{\alpha_n} f \|_{L^p(\G)}^p \right)^{1/p}, $$
where $X_jf$ is understood in the weak sense; or, as in \cite{burns1991sobolev},
$$ \mathbb{W}_{p;k}(\G) = \left \{ f:\G\to\C \mid \, X_1^{\alpha_1}\dots X_n^{\alpha_n} f \in L^p(\G) \, \; \forall \alpha \in\N_0^n, | \alpha|  \leq k\right\}, $$
with the norm
$$ \| f \|_{\mathbb{W}_{p;k}(\G)} =\left( \sum_{|\alpha| \leq k}  \| X_1^{\alpha_1}\dots X_n^{\alpha_n} f \|_{L^p(\G)}^p\right)^{1/p} .$$
On the other hand, if $\G$ is moreover a graded group and we fix a positive Rockland operator $\mathcal{R}$ of homogeneous degree $\nu$ on $\G$, it possible to consider (inhomogeneous)
Sobolev spaces of potential type $L^p_\s(\G)$ for real $s>0$ which may be defined as the domain of the fractional powers $(I+ \mathcal{R}_p)^{s/\nu}$
and are equipped with the Sobolev norm
\begin{equation}\label{norm_sob_pot}
    \| f \|_{L^p_\s(\G)} =   \|  (I+\mathcal{R}_p)^{\s/\nu}f \|_{L^p(\G)} ,
\end{equation}
where $\mathcal{R}_p$ denotes the realization of $\mathcal{R}$ in $L^p(\G)$ (see \cite [Theorem 4.3]{FR-Sobolev}).

In the Euclidean context, the spaces $W^{k,p}(\R^n)=\mathbb{W}_{p;k}(\R^n)$ are extensively used in the analysis of partial differential equations, see for instance \cite{Brezis}. On the other side, the classical potential spaces $L^p_s(\R^n)$, which correspond to take  
 $\mathcal{R}$ as the (minus) Laplacian operator $-\Delta$ 
(with $\nu=2$) on $\R^n$, were 
 introduced by N. Aronszajn and K. T. Smith 
in \cite{AS}, and also studied by A. Calderón in \cite{Calderon}.

\medskip

In this work, we develop some characterizations of these Sobolev spaces extending the classical results to 
the context of homogeneous and graded Lie groups. Our proofs are based on \textit{mean value inequalities} 
for these groups. 

\medskip




In this work, we mainly focus on extending classical results on Sobolev spaces of potential type to 
the context of graded Lie groups.
By analyzing \textit{mean value inequalities} 
for these groups, we study two different kinds of characterizations:
\begin{itemize}
\item[(i)] By Littlewood-Paley-Stein square functions (also called in the literature $g$--functions) associated with the (heat) semigroup generated by $\mathcal{R}$.
\item[(ii)] By certain functional involving first and second order differences introduced (in the Euclidean context) by  R. Strichartz in \cite{Strichartz}.
\end{itemize}

Regarding (i), the characterization of functional spaces by means of square functions is a well known tool in Harmonic Analysis. For any $\alpha \in \C$ with $\Re(\alpha)>0$, we consider the Littlewood-Paley-Stein $g$--function of order $\alpha$,
\begin{equation}\label{def_g_alpha}
   g_\alpha(f) = \left( \int_0^\infty |(t\mathcal{R})^\alpha T_t f|^2 \frac{dt}{t}  \right)^{1/2}, 
\end{equation}
where $T_t$ is the (heat) semigroup generated by $\mathcal{R}$, as defined by S. Meda in \cite{Meda}. See also \cite{Cowling}.
We will prove the following basic characterization of the Lebesgue spaces $L^p(\G)$ in terms of $g_\alpha$:
\begin{theorem}\label{g_equiv_norm}
	Let $\G$ be a graded group. For any complex number $\alpha$ with $\Re(\alpha)>0$ and $1<p<\infty$ we have that 
	$$ \| g_\alpha(f) \|_{L^p(\G)} \approx \| f \|_{L^p(\G)} $$
for every function $f\in L^p(\G)$. \footnote{Throughout this paper, 
when $\|f\|\leq C\|g\|$ for some positive constant $C$ we will use the notation $\|f\|\lesssim \|g\|$, and $\|f\|\approx \|g\|$ when the norms of $f$ and $g$ are equivalent.}
\end{theorem}
From this result, we will deduce the following characterization of the potential spaces $L^p_s(\G)$ by square functions, which is one of the main results of 
this article.

\begin{theorem}\label{G_s}
Let $\G$ be a graded group. For $f \in L^p(\G)$, $1<p<\infty$, and $0<s<\nu$, consider the square function
	$$   
G_s(f) = \left( \int_0^\infty t^{1-2s/\nu}  \left| \frac{\partial T_t f}{\partial t}(x) \right|^2 \; dt
\right)^{1/2}.  $$
Then, $f \in L^p_s(\G)$ if and only if $G_s(f) \in L^p(\G)$. 	Furthermore,
	$$ \|f \|_{L^p_s(\G)} \approx \| f \|_{L^p(\G)} +  \| G_s(f) \|_{L^p(\G)}. $$
\end{theorem}

For the case of the sublaplacian in a stratified Lie group, the semigroup $T_t$ is a symmetric 
diffusion semigroup in $L^p(\G)$. Then, as observed in \cite{CRTN}, this result follows from the results of 
S. Meda \cite{Meda} that generalize the previous results of E. Stein \cite{Stein-topics} (for the case 
of $\alpha$ being an integer). 
However, as noticed by Fischer and Ruzhansky in \cite{FR-Sobolev}, for a general 
Rockland operator in a graded Lie group, the assumptions of S. Meda do not hold as the semigroup 
generated by $\mathcal{R}$ is not contractive in $L^p(\G)$.
For this reason, we follow a direct different approach based on showing that $g_\alpha$ coincides 
with the square function $g_\phi$ associated to an approximation of the identity (given by \eqref{g-phi} 
below) for a specific choice of $\phi$. To our knowledge, this approach is new even in the setting of 
Euclidean spaces.

We remark that in \cite[Theorem 2.2]{Cardona-Ruzhansky} a different  characterization of 
$L^p(\G)$ is given using a dyadic Littlewood-Paley partition of the unit. In turn, this theorem depends on 
the generalization of the Mikhlin-Hörmander multiplier theorem in \cite{fischer2014fourier}. A similar result in the setting of 
Lie groups with polynomial growth (endowed with a system of vector fields satisfying the Hörmander condition and the associated sub-Laplacian) was given in \cite{furioli2006littlewood}. For more information on this setting see \cite{VSCC}). In general, our setting is different since the Rockland operator $\mathcal{R}$ may not be of second order. Carnot groups (like the Heisenberg group) are included in both settings.

\medskip

As for (ii), in \cite{Strichartz} R. Strichartz  gave an intrinsic characterization of the classical potential spaces $L^p_\s$ on $\R^n$ 
by using  the functional
$$ S_\s f(x) = \left( \int_{0}^{\infty} \left[ \int_{|y|<1} |f(x+ry) - f(x) | \; dy \right]^2 \; \frac{dr}{r^{1+2\s}} \right)^{1/2} $$ 
and its analog with second-order differences
$$  S^{(2)}_\s f(x) = \left( \int_{0}^{\infty} \left[ \int_{|y|<1} |f(x+ry)+ f(x-ry) - 2f(x)|  dy \right]^2  \frac{dr}{r^{1+2\s}} \right)^{1/2}.$$ Indeed, for $1<p<\infty$ and $0<\s<1$, the functional $S_\s$ characterizes the potential space $L^p_\s(\R^n)$ in the sense 
that $f\in L^p(\R^n)$ verifies that $f\in L^p_\s(\R^n)$ if and only if $S_\s f \in L^p(\R^n)$, and 
\begin{equation} \label{stric_in euclud}
   \| f \|_{L^p_\s(\R^n)}  \approx \| f \|_{L^p(\R^n)} +  \| S_\s f \|_{L^p(\R^n)}.  
\end{equation}
The same property holds for $S^{(2)}_\s$ in the range $0<\s<2$ (see \cite[Section 2.3]{Strichartz}).

The first assertion in  Strichartz's theorem (i.e. \eqref{stric_in euclud}) was later generalized in \cite{CRTN} to the setting of unimodular Lie groups
endowed with a system of vector fields satisfying the Hörmander condition 
(see \cite[Theorem 5]{CRTN}), 
and also to Riemannian manifolds. {Also, in the case of stratified groups, it was first studied in \cite{Bohnke} using 
a different technique.} 

We will prove Strichartz's characterization of Sobolev potential spaces in the context of graded Lie groups.

\begin{theorem}\label{characterization Strichartz}
Let $\G$ be a graded group.
Consider the functionals 
$$ S_\s f(x) = \left( \int_{0}^{\infty} \left[ \frac{1}{r^{\s+Q}} \int_{|y|\leq r} |f(x\cdot y)-f(x)| \; dy \right]^2 \; \frac{dr}{r} \right)^{1/2} $$
and 
\begin{equation}\label{S_\s^{(2)}}
   S^{(2)}_\s f(x) = \left( \int_{0}^{\infty} \left[ \frac{1}{r^{\s+Q}} \int_{|y|\leq r} |\Delta^2_y f(x)| \; dy \right]^2 \; \frac{dr}{r} \right)^{1/2}, 
\end{equation}
where $\Delta^2_y f(x)$ is the second order difference 
\begin{equation}\label{Delta^2_y}
  \Delta^2_y f(x)=f(x\cdot y)+f(x\cdot y^{-1})-2f(x).  
\end{equation}
Then, for every $1<p<\infty$, 
\begin{enumerate}
    \item[(i)] if $\s\in (0,1)$, $f\in L^p_\s(\G)$ if and only if $S_\s f \in L^p(\G)$,  and  
\begin{equation}\label{scrtich_estim_1}
    \| f \|_{L^p_s(\G)} \approx \| f \|_{L^p(\G)} + \| S_s f \|_{L^p(\G)};   
\end{equation}    
    \item[(ii)] if $\s\in (0,2)$, $f\in L^p_\s(\G)$ if and only if $S_\s^{(2)} f \in L^p(\G)$, and  
\begin{equation}\label{scrtich_estim_2}
    \| f \|_{L^p_s(\G)} \approx \| f \|_{L^p(\G)} + \| S_s^{(2)} f \|_{L^p(\G)} .   
\end{equation} 
\end{enumerate}
\end{theorem}

We stress that the main novelty of this result 
is the characterization by second-order differences
\eqref{scrtich_estim_2}, as similar 
results using first-order differences have previously appeared for instance in \cite{CRTN} and 
\cite{Bruno-Peloso-Vallarino} (in a slightly different context, namely that of Lie groups with polynomial growth). We also remark that an analogous result is valid for Triebel-Lizorkin spaces, which generalize Sobolev spaces. This can be deduced by following \cite{Bruno-Peloso-Vallarino} and using our second-order mean value inequalities in the setting of graded groups.

\medskip

Apart from that, it follows from \cite[Lemma 4.18]{FR-Sobolev} that $W^{\nu \ell,p}(\G)= L^{p}_{\nu \ell}(\G)$ for any $\ell \in \N$. From this fact, and an interpolation argument,
it is proved in \cite[Theorem 4.20]{FR-Sobolev} that the spaces $L^p_\s(\G)$ are in fact independent of the choice of the Rockland operator 
$\mathcal{R}$. In this article, Theorem \ref{characterization Strichartz} allows us to deduce this independence from a different approach for the range $s\in (0,2)$.

\begin{corollary}\label{indepencdencia_Rockland}
Let $\G$ be a graded group, $1<p<\infty$, and
 $\s\in (0,2)$. Then, 
the Sobolev spaces $L^p_\s(\G)$ associated to any two positive Rockland operators coincide; moreover, their norms  are
equivalent.
\end{corollary}

\medskip

As a byproduct of our investigations, we also give an explicit formula for the fractional powers of a Rockland operator $\mathcal{R}$ (Theorem \ref{Teo_R_puntual} and Corollary \ref{Coro_rock_characterization}), which generalizes a known formula for the fractional Laplacian in Euclidean spaces (see \cite[Lemma 3.2]{DPV}). We also study their 
mapping properties in the spaces $\mathbb{W}_{p;\nu}(\G)$ (Proposition \ref{Prop_cont_rockland}).

\medskip

As said before, a fundamental role in our characterization theorems is played by mean value inequalities. Indeed, a first order mean value inequality for $L^p(\G)$ (Theorem \ref{Lp_mean_value}) allows us to prove Theorem \ref{characterization W1,p}; and also we use it for proving Theorem \ref{g_equiv_norm} for $p\not= 2$. Additionally, with respect to Theorem \ref{characterization Strichartz}, 
a key difference with the computations in \cite{CRTN} is that we use a second-order mean value inequality (Theorem \ref{MeanValue_primero}), which allows for a wider range for the parameter $s$. In turn, this result depends on the second-order Taylor expansion which is available due to the special structure of the homogeneous Lie groups (and which is not available in the setting of Lie groups of polynomial growth considered in \cite{CRTN}). For this reason, we think that this is the right setting for our results. 

\medskip

{For the study of other related functional spaces, such as Besov spaces on Lie groups we refer the reader for example to  \cite{Saka} (on nilpotent Lie groups), and also to \cite{furioli2006littlewood,Sire} (in the setting of Lie groups with polynomial volume growth quipped with a second-order sub-Laplacian.)}. We remark however that, in general, our setting is different, since the Rockland operator $\mathcal{R}$ can be of higher order (see Section \ref{sec_frac_power_rock} for details). 

\medskip

This article is organized as follows. In Section \ref{section-pelim} we review some basic definitions and results. Section \ref{sec_meanvalue} 
is devoted to the study of mean value inequalities on homogeneous Lie groups. It
starts by recalling a classical mean value inequality for homogeneous groups. Then, in Subsection \ref{Lp_version}, we present a first order mean value inequality for $L^p(\G)$ (Theorem \ref{Lp_mean_value}).  This allows us to prove Theorem \ref{characterization W1,p} in Subsection \ref{charcterization_derivatives}, which is a characterization of $\mathbb{W}_{p;1}(\G)$ by using first order differences.
Finally, in Subsection \ref{subsec_2nd_mv} we develop second order mean value inequalities: Theorems \ref{MeanValue_primero} and \ref{teo_second order_for_Sobolev}. In the following three sections, we will assume $\G$ as a graded group. Section \ref{sec_frac_power_rock} deals with the fractional powers of a Rockland operator. 
In Section \ref{sec_g-function} we prove Theorems \ref{g_equiv_norm} and \ref{G_s}; and  in Section \ref{sec_strichatz} we prove Theorem \ref{characterization Strichartz}.
The proof of Theorem \ref{g_equiv_norm} holds from spectral theory if $p=2$; and for the case $p \not = 2$, the function $\phi_\alpha$ (given by  \eqref{phi_alpha}) plays a crucial role.  We include an Appendix with the aim of giving an explicit expression for this function 
in the Euclidean space.

\section{Basic definitions: Homogeneous and graded Lie groups}\label{section-pelim}

For the convenience of the reader we repeat the relevant material from  \cite{ FR-Sobolev} and \cite{BLU} 
without proofs, thus making our exposition self-contained and also for fixing notation.

Let $V$ be a real vector space. A family $\{D_\lambda\}_{\lambda>0}$ of linear operators on $V$ is
called a set of \textit{dilations} on $V$ if there are real numbers $\sigma_j>0$ and subspaces $W_{\sigma_j}$
of $V$ such that $V$ is the direct sum of the $W_{\sigma_j}$
and
$$(D_\lambda)_{|_{W_{\sigma_j}}}
= \lambda^{\sigma_j}\mathrm{Id}.$$

\begin{definition}\label{def_homog}
	A connected  simply connected Lie group $(\G,\cdot)$ is said \textit{homogeneous} if it is equipped with a family of automorphisms such that their differentials at the identity point give rice to a family of dilations on its Lie algebra $\mathfrak{g}$.
\end{definition}

It follows that every homogeneous Lie group is nilpotent. Then the exponential map $\exp_\G$ is a surjection and by abuse of notation, we will use the same notation  $\{D_\lambda\}_{\lambda>0}$ for  the group automorphisms in Definition \ref{def_homog} and for their  differentials, or simply $D_\lambda(x)=\lambda x$ for the group automorphisms. 
Having both the group and dilation structures, homogeneous Lie groups are a natural setting to generalize many questions of Euclidean harmonic analysis. 

From now on, let  $(\G,\cdot)$ be an $n$-dimensional homogeneous Lie group. Then,   there exists a basis $\{X_1, . . . , X_n\}$ of its Lie algebra $\mathfrak{g}$ and positive numbers $\sigma_1,\dots, \sigma_n$ such that 
\begin{equation}\label{basis_dil}
	D_\lambda(X_j) = \lambda^{\sigma_j}X_j \qquad \text{for all } \lambda > 0, \, j= 1,\dots, n.
\end{equation}
We called $\sigma_1,\dots,\sigma_n$ the \textit{dilations' weights}. Without loss of generality, we can assume that $1\leq\sigma_1\leq \ldots \leq \sigma_n$. The homogeneous dimension of $\G$ is defined by
\begin{equation*}
	Q=\sum_{j=1}^n \sigma_j.
\end{equation*} 
We will say that a basis $\{X_1,...,X_n\}$ of $\mathfrak{g}$ satisfying \eqref{basis_dil} is a \textit{Jacobian basis}. By fixing such a basis, the group $\G$ may be identified with $\R^n$ equipped with a polynomial law. With this identification
the unit element of $\G$ may simply be denoted $0$.  
Moreover, this yields a Lebesgue measure on $\mathfrak{g}$ and a Haar measure on $\G$ for which holds $$\int_{\G}f(\lambda x) \, dx = \lambda^{-Q}\int_{\G} f(x) \, dx \qquad \forall f\in L^1(\G).$$
Besides, any homogeneous group admits a \textit{homogeneous quasi-norm} relative to the given dilations $\{D_\lambda\}_{\lambda>0}$, which is a continuous function $|\cdot | : \G \to [0, +\infty)$
	satisfying 
	\begin{enumerate}
		\item[i)] $|x| = 0$ if and only if $x=0$; 
		\item[ii)]$|x^{-1}| = |x|$ for every $x\in\G$; \item[iii)] $|D_\lambda(x)| = \lambda|x|$ for every $x \in\G$ and $\lambda >0$.
    \end{enumerate}
The existence can be shown by defining
\begin{equation}\label{equasi_norm}
    |x|=\sum_{j=1}^n |c_j|^{1/\sigma_j} \qquad \text{ or equivalently } \qquad  |x|=\max_{1\leq j\leq n} |c_j|^{1/\sigma_j}
\end{equation}
where $x=\exp_\G(\sum_{j=1}^n c_jX_j)\in\G$.
In general, if $|\cdot|$ is a homogeneous quasi-norm in a homogeneous group $\G$, there is a constant $\rho>0$ such that 
\begin{equation}\label{triang}
 |x\cdot y|\leq \rho (|x|+|y|) \qquad \forall x\in \G.  
\end{equation}
Moreover, the topology
induced by the $|\cdot|$-balls
$$B(x,r) = \{y \in \G : \, |x^{-1}\cdot y| < r\} \qquad (x\in\G, r>0)$$
coincides with the Euclidean topology of $\R^n$. As a comment, notice that the volume  of $B(x,r)$ is $r^Q$, for every $x\in \G$. Furthermore, for our proofs, we will use polar coordinates for homogeneous groups, see \cite[Proposition 3.1.42]{FR-book}.

We define the \textit{$\G$-length} of a multi-index $\beta \in \N_0^n$ (where $\N_0=\N \cup \{0\}$) as
$$ |\beta|_\G = \langle \beta, \sigma \rangle = \sum_{i=1}^n \beta_i \sigma_i .$$
We say that $p:\G\to \R$ is a polynomial function is $f\circ\exp_G$ is a polynomial on the vector space $\mathfrak{g}$. Moreover, by abuse of notation,  can set  
$$ p(x) = \sum_{\beta\in F  } c_\beta x^\beta= \sum_{\beta\in F  } c_\beta x_1^{\beta_1}\dots x_n^{\beta_n} \qquad (F \; \hbox{finite}, \, c_\beta \in \R),$$
and its \textit{homogeneous degree} is defined as
$$ \hbox{deg}_\G(p) = \max \left\{|\beta|_{\G} : \, c_\beta  \neq 0 \right\}. $$ 
In general, we said that a real valued function $f$ (or distribution) on $\G\setminus\{0\}$  is \textit{homogeneous} of homogeneous degree
$\nu$ if
$f \circ D_\lambda = \lambda^\nu f$ for any $\lambda >0$. In particular, through the identifications stated above, the coordinate function 
$$\G \ni x=(x_1,\dots,x_n)\mapsto x_j$$
is homogeneous of homogeneous degree $\sigma_j$.

\begin{proposition}\label{composition_inv_law} \cite[Theorem 1.3.16 and Corollary 1.3.16]{BLU}
	Let $\G$ be a homogeneous Lie group. 
	Then the group multiplication and inversion  have polynomial component
	functions. Furthermore, we have
	$$(x\cdot y)_1=x_1+y_1 \qquad (x\cdot y)_k= x_k+y_k+Q_k(x,y)$$
	$$(x^{-1})_1=-x_1 \qquad  (x^{-1})_k = 
	-x_k + q_k(x) $$  
	where $Q_k$ is a sum of mixed monomials in $x$, $y$, it only 
	depends on $x_l$'s and $y_l$'s
	with $\sigma_l < \sigma_k $ and satisfies $Q_k(D_\lambda(x),D_\lambda(y))=\lambda^{\sigma_k}Q_k(x,y)$; and $q_k(x)$ is a polynomial
	function in $x$, homogeneous of homogeneous degree $\sigma_k$
	depending only on $x_l$'s
	with $\sigma_l < \sigma_k $.
\end{proposition}

By \cite[Proposition 3.1.28]{FR-book}, each vector field $X_j$ of a Jacobian basis of $\mathfrak{g}$ is homogeneous of homogeneous degree $\sigma_j$ and takes the form
\begin{equation}\label{jacobian_basis_carnot}
	X_j = \sum_{k=j}^n a_k^{(j)}(x_1,\dots,x_n) \partial_{x_k} ,
\end{equation}
where $a_k^{(j)}$ is a polynomial function of homogeneous degree $\sigma_k-\sigma_j$ and
 $a_j^{(j)} =1$. 
For any $\alpha\in\N_0^n$, we denote 
\begin{equation}\label{X}
	\mathrm{X}^\alpha:=X_1^{\alpha_1}\dots X_n^{\alpha_n}.
\end{equation}

\begin{definition}
	A connected simply connected Lie group $(\G,\cdot)$ is said \textit{graded} if its Lie algebra is endowed with a vector space decomposition as a direct sum 
	\begin{equation}\label{grad}
		\mathfrak{g}=W_1\oplus W_2 \oplus \dots \oplus W_r   \qquad \text{such that } \, 
		[W_i,W_j]\subseteq W_{i+j}.
	\end{equation}
\end{definition}

A gradation over a Lie algebra is not unique. 
Graded Lie groups are homogeneous and the \textit{natural} homogeneous
structure for the graded Lie algebra
is the same for any of the two gradations. Indeed, 
if $\G$ is a graded Lie group with a gradation $\{W_j\}$  of $\mathfrak{g}$ as in \eqref{grad}, the dilations
$$D_\lambda(X) = \lambda^{j}X \qquad \text{if }
X \in W_j$$
are automorphisms. This shows that $\G$ canonically inherits a homogeneous structure. Conversely, if the dilations $\{D_\lambda\}_{\lambda>0}$ on a homogeneous Lie algebra $\mathfrak{g}$ have eigenvalues $\lambda^{j}$
(i.e. with integer exponents), the eigenspaces $W_j$ relative to the eigenvalues $\lambda^j$ form
a gradation of $\mathfrak{g}$. In fact, 
the relevant structure for the analysis of graded Lie groups
is their natural homogeneous structure.

A graded Lie group  $\G$ such that its Lie algebra $\mathfrak{g}$ admits a  direct sum decomposition as in \eqref{grad} satisfying
$[W_1,W_{j-1}]=W_j$ if $2 \leq j \leq r$ and $[W_1,W_r]=\{0\}$ is called \textit{stratified} (or \textit{Carnot}). In this case, the first stratum $W_1$ generates $\mathfrak{g}$ as an algebra. We will write $n_1$ for the dimension of $W_1$. 

\section{Mean Value Inequalities on graded Lie groups}\label{sec_meanvalue}

Throughout this section, unless otherwise specified, we consider $\G$ a homogeneous Lie group where we fix dilations' weights $1\leq\sigma_1\leq \ldots \leq \sigma_n$, and  a Jacobian basis $\{X_1,X_2, \ldots, X_n \}$ of $\mathfrak{g}$. 

\subsection{Previous results}
In this subsection, we collect the classical mean value theorem on homogeneous Lie groups and a well-known key lemma (see Lemma \ref{lemma3.1.47}). 
For the particular case of Carnot groups see \cite[Chapter 1]{Folland-Stein}, or the book \cite{BLU}.

\begin{theorem} (Mean Value Theorem) \cite[Proposition 3.1.46]{FR-Sobolev}
\label{Teorema_MeanValue}	Let $\G$ be a homogeneous Lie group and fix a homogeneous quasi-norm $|\cdot|$ on $\G$. There exist group
	constants $C_0 > 0$ and $\eta > 1$ such that for all $f \in C^1(\G)$ and all $x, y \in\G$, we have
	\begin{equation*}
		|f(x\cdot y)-f(x)|\leq C_0\sum_{j=1}^n |y|^{\sigma_j}\sup_{|z|\leq \eta|y|}|X_jf(x\cdot z)|.
	\end{equation*}
\end{theorem}

The following lemma  is useful in the proof of  Theorem \ref{Teorema_MeanValue} and we will use it to generalize the mean value inequality on $L^p(\G)$ (see Theorem \ref{Lp_mean_value}).

\begin{lemma}\cite[Lemma 3.1.47]{FR-Sobolev}\label{lemma3.1.47} Let $\G$ be a homogeneous Lie group.
	The map $\Phi:\R^n\to \G$
	\begin{equation*}
		\Phi(t_1,\dots,t_n):=\mathrm{exp}_\G(t_1X_1)\cdot\ldots\cdot \mathrm{exp}_\G(t_nX_n)
	\end{equation*}
	is a global diffeomorphism. Moreover, fixing a homogeneous quasi-norm $|\cdot|$ on $\G$, there is a constant $C_{1}>0$ such that 
$$
\left|t_{j}\right|^{\frac{1}{\sigma_{j}}} \leq C_{1}\left|\Phi\left(t_{1}, \ldots, t_{n}\right)\right| \qquad \forall\left(t_{1}, \ldots, t_{n}\right) \in \mathbb{R}^{n}, \, j=1, \ldots, n.
$$
\end{lemma}

\begin{remark}
If $\G$ is a Carnot group, there exists $C>0$ and $m\in\N$ such that any $x\in\G$ can be express as $x=x_1\cdot x_2 \cdot\ldots\cdot x_m$ where $x_j=\exp(Z_j)$ with $Z_j$ in the first stratum $W_1$, and $|x_j|\leq C|x|$ for all $j$. See \cite[Lemma (5.1)]{Folland} or \cite[Lemma (1.40)]{Folland-Stein}.
\end{remark}

\begin{definition}\label{def_Taylor_homog} (see \cite[Definition 3.1.49]{FR-Sobolev})
	Let $(\G,\cdot)$ be a homogeneous Lie group. 
	The (left) Taylor polynomial of a suitable function $f$ centered at $x_0$ and
	of $\G$-order $k \in \N_0$ is the unique  polynomial $p$ with $\textrm{deg}_\G(p)\leq k$ such that
	for every multi-index $\alpha\in\N^{n}_0$ with $|\alpha|_\G\leq k$ satisfies
	\begin{equation}\label{Taylor_pol}
		\mathrm{X}^\alpha p(0) = \mathrm{X}^\alpha f(x_0) 
	\end{equation}
	for $\mathrm{X}^\alpha$  defined as in \eqref{X}.
	We denote $p$ as $P_{x_0,k}^f$. Besides, the remainder of order $k$ is defined as
	$$R_{x_0,k}^f(y)=f(x_0\cdot y)-P_{x_0,k}^f(y)$$
\end{definition}
In the definition, the suitable functions $f:\G\to\R$ are such that their derivatives $\mathrm{X}^\alpha f$ are continuous in a
neighbourhood of $x_0$ for $|\alpha|_\G \leq k$


\begin{theorem}\label{Th_Taylor_homog} \cite[Theorem 3.1.51]{FR-Sobolev}
	Let $\G$ be a homogeneous Lie group and  fix a homogeneous quasi-norm $|\cdot|$ on $\G$. 
	For any
	$k \in\N_0$, there exist constants $\eta, c_k > 0$ such that for all functions $f \in C^{[k]+1}(\G)$,
	where $$[k]:=\max\{|\alpha| :\  \alpha \in\N^n_0
	\text{ with } |\alpha|_\G \leq k\},$$
	and for all $x, h \in G$, we have
	\begin{equation*}
		|R^f_{x,k}(h)| \leq c_k  \sum_{|\alpha|\leq [k]+1 , \, |\alpha|_\G>k}|h|^{|\alpha|_\G} \sup_{|z|\leq\eta^{ [k]+1} |h|} |\mathrm{X}^\alpha f(x\cdot z)|
	\end{equation*}
	where $\mathrm{X}^\alpha$ is as in \eqref{X}. 
The 
constant $\eta$ comes from the Mean Value Theorem \ref{Teorema_MeanValue}. 
\end{theorem}

For stratified groups, Theorem \ref{Th_Taylor_homog} can be simplified (see \cite[Corollary 20.3.2,Theorem 20.3.3] {BLU}).

\subsection{An Lp version of the mean value theorem}\label{Lp_version}

As in \cite{burns1991sobolev}, consider the Sobolev space defined in terms of the first-order derivatives 
$$\mathbb{ W}_{p;1}(\G) = \left \{ f:\G\to\C \mid \,  X_1f, \dots, X_n f \in L^p(\G)  \right\}, $$
where $X_jf$ is understood in the weak sense.
The following result is an $L^p$ version of the mean value theorem on homogeneous Lie groups. It extends a result known in Euclidean spaces 
\cite[Proposition 9.3]{Brezis}. See also \cite[Lemma 10]{Saka} in the context of Carnot groups, in which case only derivatives in the first stratum are needed on the right hand side.

\begin{theorem}\label{Lp_mean_value}
	Let $\G$ be and homogeneous Lie group, $1\leq p<\infty$, and  $f\in \mathbb{W}_{p;1}(\G)$. Then, 
	\be \left( \int_{\G} |f(x\cdot y)-f(x)|^p \; dx \right)^{1/p} \lesssim
	\sum_{j=1}^n |y|^{\sigma_j}  \| X_j f \|_{L^p(\G)}. \label{eqn-mean-value-Lp} \ee
\end{theorem}

\begin{proof}
    Let $f\in C^1(\G)$ with compact support and then proceed by density.
	Following \cite[proof of Proposition 3.1.46]{FR-book} if $y= \exp_{\G}(tX_j)$, we have that
	\begin{align*}
		f(x \cdot y) - f(x) 
		&= \int_0^t \partial_{s^\prime=s}\left\{ f(x \cdot \exp_{\G} (s^\prime X_j)) \right\} \; ds \\
		&= \int_0^t \partial_{s^\prime=0}\left\{ f(x  \cdot \exp_{\G}(s X_j)  \cdot \exp_{\G}(s^\prime X_j)) \right\} \; ds \\
		&= \int_{0}^t X_j f(x  \cdot \exp_{\G}(sX_j)) \; ds \\
		&= \int_{0}^1 X_j f(x \cdot  \exp_{\G}(utX_j)) \, t \; du
	\end{align*}
Then using Jensen inequality,
	$$ |f(x\cdot y) - f(x)|^p \leq |t|^p \int_{0}^1 |X_j f(x \cdot \exp_{\G}(utX_j))|^p  \; du .$$
	Integrating over $\G$ and using Fubini--Tonelli theorem,
	\begin{align*}
		\int_{\G}  |f(x \cdot y) - f(x)|^p \; dx &\leq |t|^p \int_{\G} \left[ \int_{0}^1  |X_j f(x \cdot\exp_{\G}(utX_j))|^p \; du \right] \; dx \\
		&= |t|^p \int_0^1 \left[ \int_{\G} |X_j f(x \cdot\exp_{\G}(utX_j))|^p \; dx  \right] \; du. 
	\end{align*}
	Now, for each fixed $t$, we make the change of variables $z=x  \cdot \exp_{\G}(utX_j)$, and using the invariance of the Haar measure 
	we get
	\begin{align*}
		\int_{\G}  |f(x\cdot y) - f(x)|^p \; dx &\leq |t|^p \int_0^1 \left[ \int_{\G} |X_j f(z)|^p \; dz  \right] \; du \\
		&= |t|^p \int_{\G} |X_j f(z)|^p \; dz .  
	\end{align*}
	As observed in \cite{FR-book}, $|y|=|\exp_{\G}(tX_j)| = |t|^{1/\sigma_j} |\exp_{\G}(X_j)|$. 
	Setting
	$$ C_2 := \max_{k=1,2,\ldots,n} |\exp_{\G}(X_k)|^{-\sigma_k} $$
	we have that
	$$ |t| \leq C_2 |y|^{\sigma_j} $$
	and hence
	\begin{align*}
		\left(   \int_{\G}  |f(x \cdot y) - f(x)|^p \; dx \right)^{1/p} &\leq C_2 |y|^{\sigma_j} \| X_j f \|_{L^p(\G)}  .
	\end{align*}
	
	For the general case, we observe that by  Lemma \ref{lemma3.1.47} each $y \in \G$ 
	can be written uniquely as $y=y_1 \cdot y_2 \cdots y_n$ with $y_j=\exp_{\G}(t_j X_j)$, and we have the estimate
	$$|t_j|^{1/\sigma_j} \leq C_1 |y|, $$
 hence
	$$ |y_j| = |t_j|^{1/\sigma_j} |\exp_{\G}(X_j)| \leq C_1 C_3 |y| $$
	where
	$$ C_3= \max_{k=1,\ldots, n} |\exp_{\G}(X_j)| $$ 
	We consider the sequence of points 
	$$ p_0(x)=x, \quad  p_k(x)= x \cdot y_1 \cdot y_2 \cdot\ldots\cdot   y_k \; \quad  \hbox{for} \; k=1,2,\ldots,n.$$ 
	Hence,
	$$|f(x\cdot y)-f(x)| \leq \sum_{j=1}^n |f(p_j(x))- f(p_{j-1}(x))|.   $$
	We integrate over $\G$ with respect to the variable $x$. By Minkowski inequality
	\begin{equation}\label{auux1}
	\left( \int_{\G} |f(x\cdot y)-f(x)|^p \; dx \right)^{1/p}  \leq \sum_{j=1}^n \left( \int_{\G}  |f(p_j(x)) - f(p_{j-1}(x))|^p \;dx \right)^{1/p}.
	\end{equation}
	We consider one of the terms in the right hand side. Making the change of variables $ z= x\cdot y_1\cdot y_2\cdot \ldots \cdot y_{j-1} $ and using the invariance of the Haar measure, 
	we have that
	\begin{align}\label{auux2}
		\left( \int_{\G}  |f(p_j(x)) - f(p_{j-1}(x))|^p \; dx \right)^{1/p}  &= \left( \int_{\G} |f(z \cdot y_j) - f(z)|^p \; dz\right)^{1/p} \notag \\ 
		&\leq |y_j|^{\sigma_j} \| X_j f \|_{L^p(\G)} \\ 
		&\leq (C_1C_3 |y|)^{\sigma_j} \| X_j f \|_{L^p(\G)} \notag .
	\end{align}
	Summing up all the contributions we get \eqref{eqn-mean-value-Lp}.
\end{proof}

\begin{remark}\cite[Lemma 10]{Saka}\label{remark: saka}
If $\G$ is a Carnot group, and $1\leq p< \infty$ one can obtain
	\be \left( \int_{\G} |f(x\cdot y)-f(x)|^p \; dx \right)^{1/p} \lesssim
	|y|\sum_{j=1}^{n_1} \| X_j f \|_{L^p(\G)}.\ee
taking the sum only in the first stratum.
We recall that the case $p=\infty$ is due to G. Folland
(see \cite[Proposition (5.4)]{Folland}). 
\end{remark}

\subsection{A characterization of the Sobolev spaces defined in terms of first-order derivatives}\label{charcterization_derivatives}

Now we have all the ingredients to prove  a characterization of $\mathbb{W}_{p;1}(\G)$ by means of the 
translations on $\G$.

\begin{theorem}\label{characterization W1,p}
	Let $\G$ be a homogeneous Lie group {with dilations' weights $1\leq\sigma_1\leq \ldots \leq \sigma_n$\footnote{See Section \ref{section-pelim}}}. When $1<p<\infty$, $f \in \mathbb{W}_{p;1}(\G)$ if and only if $f \in L^p(\G)$ and there exists a constant $C=C(f)$ depending of $f$ such that 
{		\begin{equation*}\label{newteo11}
\left( \int_{\G} |f(x\cdot y)-f(x)|^p \; dx \right)^{1/p} \leq C \; \sum_{j=1}^n  |y_j|^{\sigma_j}
 	\end{equation*}
where  $y\in \G$ is (uniquely) written as 
$y=y_1 \cdot y_2  \cdot\ldots\cdot  y_n$ with $y_j=\exp_{\G}(t_j X_j)$. }
\end{theorem}

\begin{proof}

The {direct part of our assertion} follows from {the proof of  Theorem \ref{Lp_mean_value} (in fact, from \eqref{auux1} and \eqref{auux2})} taking $C(f)=n\max_{1\leq j\leq n}  \| X_j f \|_{L^p(\G)}$.  

For the converse, let $f\in L^p(\G)$. Consider $\varphi\in C^1(\G)$ with compact support  and let $y\in\G$. By the invariance of the Haar measure, we have
\begin{equation*}
    \int_\G \left(f(x\cdot y)-f(x)\right)\varphi(x) \, dx= \int_\G f(x)\left(\varphi(x\cdot y^{-1})-\varphi(x)\right) \, dx.
\end{equation*}
Also, using Hölder inequality and the hypothesis, one obtains
{ 
$$\left|\int_\G \left(f(x\cdot y)-f(x)\right)\varphi(x) \, dx\right|\lesssim\sum_{j=1}^n |y_j|^{\sigma_j} \|\varphi\|_{L^{p'}(\G)} $$
where $\frac{1}{p}+\frac{1}{p'}=1$, and $y \in \G$ is written uniquely as $y=y_1 \cdot y_2 \cdot\ldots\cdot y_n$ with $y_j=\exp_{\G}(t_j X_j)$ (Lemma \ref{lemma3.1.47}).  
For each $1\leq j\leq n$, consider $y^{-1}=\exp_{\G}(tX_j)$ for $t\in\R$ with $|t|\leq 1$. Therefore,
\begin{equation}\label{new1}
    \left|\int_\G f(x)\left(\varphi(x\cdot y^{-1})-\varphi(x)\right) \, dx\right|\lesssim |y|^{\sigma_{j}}\|\varphi\|_{L^{p'}(\G)}.
\end{equation} 
Also
$$|y|=|y^{-1}|=|\exp_\G(tX_{j})|=|t|^{1/\sigma_{j}}|\exp_\G(X_{j})|\lesssim|t|^{1/\sigma_{j}},$$
and so
$$ |y|^{\sigma_{j}}\lesssim|t|.$$
 Dividing by $t$ in \eqref{new1} and letting $t \to 0$, we obtain}
\begin{equation}\label{aux-brezis}
    \left|\int_\G f(x) \, X_j\varphi (x) \, dx\right|\lesssim \|\varphi\|_{L^{p'}(\G)}. 
\end{equation}
Note that \eqref{aux-brezis} holds for every vector field in the given Jacobian basis of $\G$. In order to end the proof we will show that for $f\in L^p(\G)$ satisfying \eqref{aux-brezis} for every $\varphi\in C^1(\G)$ with compact support and for every $1\leq i\leq n$, it holds that $f\in \mathbb{W}_{p;1}(\G)$.

For each $1\leq i\leq n$, consider the linear functional
\begin{equation}
    \varphi\mapsto \int_\G f(x) \, X_i\varphi(x) \,  dx 
\end{equation}
defined in the space of $C^1(\G)$ functions with compact support, which a dense subspace of $L^{p'}(\G)$ (since $p'<\infty$), and it is continuous for the
$L^{p'}(\G)$ norm (by \eqref{aux-brezis}). Therefore, it extends to a bounded linear functional defined on all of $L^{p'}(\G)$. 
By Riesz representation theorem, there exists $g_i \in L^p(\G)$ such that 
\begin{equation*}
    \int_\G f(x) \, X_i\varphi(x) \,  dx =  \int_\G \varphi(x) g_i(x) \, dx 
\end{equation*}
for every $\varphi\in L^{p'}(\G)$. Thus, each weak derivative $X_if$ is $-g_i$ which belongs to $L^p(\G)$ and so $f\in \mathbb{W}_{p;1}(\G)$.
\end{proof}

\subsection{Second order mean value inequalities}\label{subsec_2nd_mv}
In this subsection we prove two second order mean value inequalities, one is for smooth functions and the other is an $L^p$ version.

\begin{theorem}\label{MeanValue_primero}
	Let $\G$ be a homogeneous Lie group and fix a homogeneous quasi-norm $|\cdot|$ on $\G$. There exist constants $\eta>0$  such that 
	\begin{equation}\label{2nd_order_beta} 
		|\Delta^2_y f(x)| \lesssim \sum_{|\beta|\leq 2, \,  |\beta|_\G>1}|y|^{|\beta|_\G}\sup_{|z|\leq \eta|y| }|X^\beta f(x \cdot z)|
	\end{equation}
	for all $f \in C^{2}(\G)$ and every $x,y \in \G$, where $\Delta^2_y f$ is the second order difference given by \eqref{Delta^2_y}. The 
constant $\eta$ comes from the Mean Value Theorem \ref{Teorema_MeanValue}. 
\end{theorem}
\begin{proof}
Let $P_{x,1}^f$ be the left-Taylor polynomial od $f$ centered at $x$ with homogeneous degree less or equal to $1$. If we compute 
$P_{x,1}^f(y)+P_{x,1}^f(y^{-1})$ we get 
$$P_{x,1}^f(y)+P_{x,1}^f(y^{-1})=2f(x)$$
since the monomials of degree $1$ annihilates (notice that in this case, the homogeneous degree coincides with the usual degree, since $1\leq \sigma_1\leq \sigma_j$ for all $j$).
Indeed, in general if $q:\G\to \R$ is a monomial of degree $1$, let $y=\exp_\G(Y)$ and so
$$q( y^{-1} ) = q( (\exp_\G(Y))^{-1} ) =q( \exp_\G(-Y) ) = - q( \exp_\G(Y) ) = - q( y )$$
since $q\circ\exp_\G$ is a  polynomial in the usual sense.

Hence,
\begin{align*}
    \Delta_y^2f(x)&=f(x\cdot y)+f(x\cdot y^{-1})-2f(x)\\
    &=f(x\cdot y)-P_{x,1}^f(y)+f(x\cdot y^{-1})-P_{x,1}^f(y^{-1})\\
    &=R_{x,1}^f(y)+R_{x,1}^f(y^{-1})
\end{align*}
Applying Theorem \eqref{Th_Taylor_homog} with  $k=1$ (and so $[1]=1$), and using $|y|=|y^{-1}|$, we obtain
\begin{align*}
    |\Delta_y^2f(x)|&\leq|R_{x,1}^f(y)|+|R_{x,1}^f(y^{-1})|\\
    &\lesssim \sum_{|\beta|\leq [1]+1, |\beta|_\G>1}|y|^{|\beta|_\G}\sup_{|z|\leq \eta|y| }|X^\beta f(x \cdot z)|+\\
    &\qquad +\sum_{|\beta|\leq [1]+1, |\beta|_\G>1}|y^{-1}|^{|\beta|_\G}\sup_{|u|\leq \eta|y^{-1}| }|X^\beta f(x \cdot u)|\\
&=2 \sum_{|\beta|\leq 2, \,  |\beta|_\G>1}|y|^{|\beta|_\G}\sup_{|z|\leq \eta|y| }|X^\beta f(x \cdot z)|
\end{align*}
\end{proof}

We finish this subsection with a second order $L^p$ mean value inequality.

\begin{theorem}\label{teo_second order_for_Sobolev}
Let $\G$ be a homogeneous Lie group. For for $1\leq p<\infty$ if $f\in \mathbb{W}_{p;2}(\G)$, then
\begin{equation*}
	\|\Delta_y^2f\|_{L^p(\G)}\leq \max\{|y|^{2}, \, |y|^{2\sigma_n}\} \sum_{j,k=1}^n\left\|X_jX_kf\right\|_{L^p(\G)}, 
\end{equation*}	
and
\begin{equation}\label{desig_2do_orden}
|\Delta_y^2f(x)|\lesssim \max\{|y|^{2}, \, |y|^{2\sigma_n}\} \sum_{j,k=1}^n\left\|X_jX_kf\right\|_{L^\infty(\G)} \qquad \forall x\in \G     
\end{equation}

\end{theorem}

\begin{proof}
Fix $y\in \G$. 
Since $\G$ is connected and nilpotent, the exponential map is a global diffeomorphism.  In particular,  $y$ is of the form $y=\exp_\G(Y)$ for some $Y=\sum_{j=1}^n c_jX_j\in \mathfrak{g}$. 
We have $y^{-1}=\exp_\G(-Y)$. Applying the Fundamental Theorem of Calculus and differentiating in the manifold $\G$ as in the proof of Theorem \ref{Lp_mean_value} we have
\begin{align*}
 	\left[f(x \cdot y) - f(x)\right] \, + \,  & \left[f(x\cdot y^{-1})-f(x)\right]\\
	=&\int_0^{1} \frac{d}{dt}_{|_{t=s}}[f(x\cdot \exp_{\G}(tY))+
 	f(x\cdot \exp_{\G}(-tY))] \, ds\\
	=&\int_0^1 Yf(x\cdot \exp_{\G}(sY))-Yf(x\cdot \exp_{\G}(-sY))\, ds.
\end{align*}    
Repeating the argument we obtain    
\begin{align*}
	\Delta_y^2f(x)
	=&\int_0^1\int_{-s}^s Y^2f(x\cdot \exp_{\G}(zY)) \, dz\, ds
	=\int_{-1}^1\int_{|z|<1} Y^2f(x\cdot \exp_{\G}(zY)) \, ds\, dz\\
	=&\int_{-1}^1(1-|z|)\, Y^2f(x\cdot \exp_{\G}(zY)) \, dz.
\end{align*}
For $p=\infty$, we have
$$|\Delta_y^2f(x)|\leq 2 \|Y^2f\|_{L^\infty(\G)} \qquad \forall x\in\G.$$
For $1\leq p<\infty$, integrating on $\G$ and using the invariance of the Haar measure we have,
\begin{align*}
	\left(\int_\G |\Delta_y^2f(x)|^p\, dx\right)^{1/p}
	&\leq \int_{-1}^1(1-|z|)\left(\int_{\G}|Y^2f(x\cdot \exp_{\G}(zY))|^p \, dx\right)^{1/p}\, dz\\
	&=\int_{-1}^1(1-|z|)\left(\int_{\G}|Y^2f(x)|^p \, dx\right)^{1/p}\, dz\\
	&=\left(\int_{\G}|Y^2f(x)|^p \, dx\right)^{1/p}.
\end{align*}
Since $y=\exp_\G(Y)$ with $Y=\sum_{j=1}^nc_jX_j$, we have
\begin{align*}
	|Y^2f(x)|&= \left|\left(\sum_{j=1}^n c_jX_j\right)\left(\sum_{k=1}^{n}c_kX_k\right)f(x) \right|\leq\max_{1\leq j\leq n}|c_j|^2\sum_{j,k=1}^n|X_jX_kf(x)|.
\end{align*}
So far we have only used the fact that $\G$ is nilpotent.
If $\G$ is a homogeneous group, considering a homogeneous quasi-norm as in \eqref{equasi_norm} and separating into two cases, $|y|\leq 1$ and $|y|>1$, we obtain \eqref{eqn-mean-value-Lp}. In fact, consider for example the homogeneous quasi-norm
$$|y|=\max_{1\leq j\leq n}|c_j|^{1/\sigma_j}.$$
If $|y|\leq1$, then $|c_{j}|\leq 1$ for all $1\leq j\leq n$, and since $\sigma_j\geq 1$ for every $1\leq j\leq n$,
$$|c_{j}|=\left(|c_{j}|^{1/\sigma_j}\right)^{\sigma_j}\leq |c_{j}|^{1/\sigma_j} \qquad \forall \, 1\leq j\leq n.$$
So,
\begin{align*}
    |Y^2f(x)|&\leq\max_{1\leq j\leq n}|c_j|^2\sum_{j,k=1}^n|X_jX_kf(x)|\leq\left(\max_{1\leq j\leq N}|c_j|^{1/\sigma_j}\right)^{2}\sum_{j,k=1}^n|X_jX_kf(x)|\\
    &= |y|^{2}\sum_{j,k=1}^n|X_jX_kf(x)|.
\end{align*}
If $|y|>1$, then
$$J=\{j\in\{1,\ldots,n\} :\, |c_j|>1\}\not=\emptyset,$$
and $1\leq\sigma_1\leq\ldots \leq \sigma_n$, for every $j\in J$  we have 
$$|c_{j}|=\left(|c_{j}|^{1/\sigma_j}\right)^{\sigma_j}\leq \left(|c_{j}|^{1/\sigma_j}\right)^{\sigma_n}. $$
So,
\begin{align*}
    |Y^2f(x)|&\leq\max_{1\leq j\leq n}|c_j|^2\sum_{j,k=1}^n|X_jX_kf(x)|=\max_{j\in J}|c_j|^2\sum_{j,k=1}^n|X_jX_kf(x)|\\
    &\leq\left(\max_{j\in J}|c_j|^{1/\sigma_j}\right)^{2 \sigma_n}\sum_{j,k=1}^n|X_jX_kf(x)|=\left(\max_{1\leq j\leq n}|c_j|^{1/\sigma_j}\right)^{2 \sigma_n}\sum_{j,k=1}^n|X_jX_kf(x)|\\
    &= |y|^{2\sigma_n}\sum_{j,k=1}^n|X_jX_kf(x)|.
\end{align*}
\end{proof}

\section{Fractional powers of a positive Rockland operator}\label{sec_frac_power_rock}

\subsection{Preliminaries: Positive Rockland operators and  heat kernel}

In what follows  we will consider the powers of a Rockland
operator in order to define the potential Sobolev spaces. Rockland
operators generalize sub-Laplacians to
the non-stratified but still homogeneous (graded) setting. Although the definition of Sobolev spaces uses a fixed positive Rockland operator, they are actually independent from this choice (in this article we give an alternative proof of this fact, see Corollary \ref{indepencdencia_Rockland}). These spaces generalize the classical functional spaces on the Euclidean space $\R^n$ introduced in
\cite{AS} and \cite{Calderon}. 

On a Carnot group on $\G$, if we
consider Jacobian generators $X_1,X_2, \ldots, X_{n_1}$ (that is, a basis of the first stratum), the associated sub-Laplacian is
\begin{equation}\label{sub-lap}
   \Delta_{\G} = -\sum_{j=1}^{n_1} X_j^2. 
\end{equation}
Any sub-Laplacian on a stratified Lie group is a Rockland operator
of homogeneous degree $2$ (see \cite{Folland} for details). Moreover, \eqref{sub-lap} is a positive operator.

On homogeneous Lie groups, Rockland operators are left-invariant hypoelliptic differential operators that are homogeneous of positive degree.  The existence of a Rockland differential operator on a homogeneous Lie group
implies that the group must admit a graded structure (see \cite[Lemma 2.2]{ter1997spectral}). Moreover, any graded Lied groups admits a Rockland operator. For example,
\begin{equation}\label{1stRockl}
 \sum_{j=1}^{n}(-1)^{\frac{\sigma_{0}}{\sigma_{j}}} X_{j}^{2 \frac{\sigma_{0}}{\sigma_{j}}}   
\end{equation}
  is a Rockland operator on a graded Lie group $\G$ where $\{X_j\}_{j=1}^n$ is a Jacobian basis, 
 $\{\sigma_j\}_{j=1}^n$ are the dilations' weights and   $\sigma_0$ is any common multiple of $\sigma_1,\dots,\sigma_n$ (see \cite[Corollary 4.1.10]{FR-book}). As a historical note, the concept of a Rockland operator initially appeared in the the seminal paper on hypoellipticity on the Heisenberg group by C. Rockland  \cite{rockland1978hypoellipticity}, and was later extended to the setting of nilpotent groups by B. Helffer and J. F. Nourrigat in the articles   \cite{helffer1978hypoellipticite, helffer1979caracterisation}. As for harmonic analysis related to  hypoelliptic operators on arbitrary simply connected nilpotent Lie groups we can mention for example the article \cite{nagel1990fundamental}.

From now on, let $\G$ be a graded Lie group. 
A positive Rockland operator $\mathcal{R}$ is an element in the univerSsal enveloping algebra of $\mathfrak{g}$, and from the  
Poincaré-Birkhoff-Witt theorem $\mathcal{R}$ will have some expression of the form
$$ \mathcal{R} = \sum_{\alpha \in \N_0^n,|\alpha|_\G=\nu} c_\alpha \, \mathrm{X}^\alpha $$
for some $c_\alpha \in \mathbb{C}$. When the coefficients $c_\alpha$ are real, like in the example in 
\cite[Lemma 4.1.8]{FR-book}, the heat kernel will be real-valued. We shall assume that. Moreover, we are interested in Rockland operators  that are formally self-adjoint as an element of the universal enveloping algebra $\mathfrak{u}(\mathfrak{g})$, and so we will work directly with their self-adjoint extensions on  $L^2(\G)$ (see \cite[Proposition 4.1.15 and Corollary 4.1.16]{FR-book}). Besides, since every graded groups admits a positive Rockland operator (for e.g. \eqref{1stRockl} is a positive operator), then  we fix $\mathcal{R}$ to be a self-adjoint positive operator. It can be seen that the homogeneous degree $\nu$ of $\mathcal{R}$ is always even.

There exists a projection-valued measure  $P_\lambda$ such that (see \cite[Corollary 4.1.16]{FR-book}) 
\begin{equation}\label{frac_calculus_esp_th}
\mathcal{R} = \int_0^\infty \lambda \; dP_\lambda .    
\end{equation}
Then, consider the heat semigroup  $\{T_t\}_{t>0}$
defined by the functional calculus 
$$T_tf:=\int_0^\infty e^{-t\lambda} \, dP_\lambda \qquad (t>0)$$
as in \cite[Section 4.2.2]{FR-book}. It is a contraction semigroup of operators on $L^2(\G)$. 
The corresponding convolution kernel $h_t$, $t > 0$, that is such that for every $f\in L^2(\G)$
$$ T_t f(x) = f*h_t(x)=\int_{\G} f(y) h_t(y^{-1}\cdot x) \; dy \qquad (x\in\G),$$
is called the heat kernel associated to $\mathcal{R}$.
The existence of the heat kernel has been proved by Folland and Stein 
\cite[Theorem 4.25]{Folland-Stein}.
Its main properties are summarized in the following theorems. Previous studies in this regard can be found in \cite{dziubanski1989semigroups,  hebisch1989sharp} where the particular case \eqref{1stRockl} was considered. 

\begin{theorem}\cite[Theorem 4.2.7]{FR-book}\label{ht_properties}
	Then heat kernels $h_t$ associated with a positive self-adjoint Rockland operator of homogeneous degree $\nu$ on a graded Lie group $\G$ satisfy the following properties:
		\begin{enumerate}
		\item $h_t * h_s= h_{t+s} $ for all $s,t>0$, and each function $h_t$ is Schwartz;
		\item 
		$h_{r^\nu t}(rx)=r^{-Q} h_t(x) \; \text{for all} \;  x\in \G, \; r,t>0$. 
		In particular, 
			$h_t(x) = t^{-\frac{Q}{\nu}} h_1(t^{-\frac{1}{\nu}}x)$,
		and also 
		$\mathrm{X}^\alpha h_t(x) = r^{Q-|\alpha|_\G} \mathrm{X}^\alpha h_{r^{\nu}t}(rx)$; 
		\item $h_t(x)=h_t(x^{-1})$ for all $x\in \G$, $t>0$;
		\item $\int_ \G h_t(x) \, dx=1$ for all $t>0$.
	\end{enumerate}
\end{theorem}

\begin{theorem} \cite[Theorem 10 and Remark 11]{DHZ} or \cite[Theorem 1.1]{ATER} \label{theorem-DHZ}
	Let $h_t$ be the heat kernels associated with a positive self-adjoint Rockland operator of homogeneous degree $\nu$ on a graded Lie group $\G$.
	There exist a constant $c>0$ such that for every multi-index $\beta$ and $s\in\N_0$ there exists a constant $C_{\beta,s}$ such 
	that
	$$ |\partial_{t}^s X^\beta h_t(x)| \leq C_{\beta,s} \; t^{-s-Q/\nu-|\beta|_\G/\nu} \; 
	\exp\left(-c |x|^{\frac{\nu}{\nu-1}}/ t^{1/(\nu-1)}   \right) $$
for every  $X\in\mathfrak{g}$ and  $x\in\G$.
\end{theorem}

A similar result holds in the setting of Lie groups of polynomial growth (see \cite{VSCC}, Theorem VIII.2.9).

For $1 \leq p < \infty$, we denote by $\mathcal{R}_p$ the realization of $\mathcal{R}$ on $L^p(\G)$, that is, 
the operator such that $-\mathcal{R}$ is the infinitesimal
generator of the semigroup of operators $T_t$ on $L^p(\G)$ for $p\in[1,\infty)$.  In particular, $\mathcal{R}_2=\mathcal{R}$ (the self-adjoint extension on $L^2(\G)$). It can be proved that $T_t(f)= f*h_t$  defines a strongly continuous semigroup on $L^p(\G)$ for $1\leq p<\infty$, which is also
equibounded, as a consequence of Young's inequality and property (2) in Theorem \ref{ht_properties}:
$$\|f*h_t\|_{L^p(\G)}\leq \|h_1\|_{L^1(\G)}\|f\|_{L^p(\G)} \qquad \forall \, t>0, \quad \forall \, f\in L^p(\G).$$
However, as is pointed out by V. Fischer and M. Ruzhansky in \cite{FR-Sobolev}, the corresponding heat semigroup might neither be a contraction on $L^p(\G)$ spaces nor preserve positivity; this shows that the heat semigroup of $\mathcal{R}$ might not be sub-markovian and this rules out many of the techniques used in the case of sub-Laplacians.
For further details see \cite[Section 3.1]{FR-Sobolev}.

\subsection{Fractional the powers of a positive Rockland operator and Sobolev spaces of potential type}

Now we will consider the fractional the powers $(\mathcal{R}_p)^\alpha$   
as defined in the Balakrishnan--Komatsu theory (see \cite[Theorem 4.3.6]{FR-book}). 
See also \cite{MCSA} for a detailed exposition of this theory. In particular, from this book we will recall a general result that allows us to express these powers in terms of the heat semigroup 
$T_t$. For the Euclidean case see for example \cite{ten}. We remark that in Euclidean spaces, the fractional powers of minus the Laplacian operator are known as the 
fractional Laplacian operator, which has received a lot of attention in recent years. 

\begin{theorem}\label{Teo_MCSA} \cite[Theorem 4.4]{Komatsu2}
Let $-A$ generate an equibounded $C_0$-semigroup $P_t$ in a Banach space $X$. Fixed $\alpha\in\C$,
let $m> \Re(\alpha)$ be a positive integer and 	$$ K_{\alpha,m} = \int_0^\infty t^{-\alpha-1} (1-e^{-t})^m \; dt \neq 0. $$
If there exists a sequence $\varepsilon_j \to 0$
such that the limit
$$ h=\lim_{j \to \infty} \frac{1}{K_{\alpha,m}} 
\int_{\varepsilon_j}^\infty t^{-\alpha-1} (I-P_t)^m \; f \; dt $$ 
exists in the weak sense, then $f \in \text{Dom}(A^\alpha)$ and $h=A^\alpha f$. 
Conversely,	if $f\in \text{Dom}  (A^\alpha)$ then there exists
	the strong limit
	$$ A^\alpha f = \lim_{\varepsilon \to 0} \frac{1}{K_{\alpha,m}} 
\int_\varepsilon^\infty t^{-\alpha-1} (I-P_t)^{m} f \; dt .$$
 	\label{thm-fractional-powers}
\end{theorem}

\begin{remark}
	When $0<\Re(\alpha)<1$,
	$ \, K_{\alpha,1}=-\Gamma(-\alpha) \neq 0$. 
\end{remark}

We will particularize the previous theorem to the case $A=\mathcal{R}_p$ and $\alpha\in\C$ with $\Re(\alpha)=s/\nu$. We first recall that if $1\leq p<\infty$ and $s>0$, the Sobolev potential space $L^p_s(\G)$  
 is characterized by 
 \begin{equation}\label{Th 4.3 FR inclusion}
  L^p_s(\G) = \hbox{Dom}\left((I + \mathcal{R}_p)^\alpha\right)=\hbox{Dom}\left((\mathcal{R}_p)^\alpha\right) \subset L^p(\G)   
 \end{equation}
endowed with the norm $\|\cdot\|_{L^p_s(\G)}$ given by \eqref{norm_sob_pot}. Moreover, $\|\cdot\|_{L^p_s(\G)}$ is equivalent to the following norms 
\begin{equation}\label{Theorem 4.3 FR-book}
  \|f\|_{L^p(\G)}+\|\left((I + \mathcal{R}_p)^{s/\nu}\right)f\|_{L^p(\G)} \quad \text{or} \quad \|f\|_{L^p(\G)}+\|( \mathcal{R}_p)^{s/\nu}f\|_{L^p(\G)}.
\end{equation}
We refer the reader to Definition 4.2, and Theorem 4.3 in \cite{FR-Sobolev}.
Now, from Theorem \ref{Teo_MCSA} we have
	\begin{equation}\label{app_teo_A_alpha}
   \mathcal{R}^\alpha_p f = \frac{-1}{\Gamma(-\alpha)}\,  \lim_{\varepsilon \to 0}  \int_{\varepsilon}^\infty t^{-\alpha-1} (f-T_t f) \; dt \qquad \forall f\in L_s^p(\G). 
\end{equation}

Some useful properties of the fractional powers of a positive Rockland operator follows easily from 
the representation in Theorem \ref{thm-fractional-powers}:

\begin{corollary}
Let $\mathcal{R}$ be a positive Rockland operator of homogeneous degree $\nu$ defined in a graded Lie group $\G$. Let $1\leq p<\infty$ and  $\Re(\alpha)>0$.
\begin{itemize}
\item[i)] Scaling property:  
\be \mathcal{R}_p^\alpha \left(f(D_r(x))\right) = r^{\alpha \nu} (\mathcal{R}_p^\alpha f)\left( D_r(x) \right) \quad \text{ for all } r>0, x\in\G. 
\label{fractional-powers-scalling}
\ee      
\item[ii)] Let $f \in L^1(\G)$ and $g \in L^p(\G)$. We have that 
\be \mathcal{R}_p^\alpha(f*g) = f* \mathcal{R}_p^\alpha(g)  .
\label{scalling-fractional-powers}
\ee
\end{itemize}
\end{corollary}
	
\begin{proof}
The scaling property follows from the corresponding property of the heat kernel (Theorem \ref{ht_properties} (2)). 
The second property follows from the associativity of convolution as
$$ T_t(f*g) = (f*g)*h_t = f*(g*h_t) = f*(T_t \; g) .$$
\end{proof}
	
In what follows we will investigate some relations between the different Sobolev spaces defined in the introduction. That is, we will recall and study some  
mapping properties of $\mathcal{R}_p^\alpha$ in relation with the spaces $\mathbb{W}_{p;k}(\G)$ (and with $W^{k,p}(\G)$ in case of Carnot groups).

Let $1<p<\infty$. In the Euclidean case,   $L^p_m(\R^n)=\mathbb{W}_{p;m}(\R^n)=W^{m,p}(\R^n)$ for every positive integer $m$ (see for e.g. \cite[Theorem 2.2]{Mizuta}). However, in the case of groups, the relations between the two Sobolev
spaces are more involved. It is known that when $s$ is a positive multiple of $\nu$ then $W^{s,p}(\G)= L^{p}_{s}(\G)$  (see \cite[Lemma 4.18]{FR-Sobolev}).  In particular, if  $0<s<\nu$, from \cite[Theorem 1.1]{FR-Sobolev} one can deduce that
$W^{\nu,p}(\G)\subset L^{p}_s(\G)$. Moreover,
from \eqref{Th 4.3 FR inclusion},
$W^{\nu,p}(\G)\subset \hbox{Dom}(\mathcal{R}_p^{s/\nu})$ (this holds for $1\leq p<\infty$), and then 
using Semigroup Theory (see for e.g. \cite{Pazy}), it can also be seen that when $0<s<\nu$,  $\mathcal{R}_p^{s/\nu}$ is a bounded operator from $W^{\nu,p}(\G)$ to $L^p(\G)$, that is 
\begin{equation}\label{desig_aux} 
\| \mathcal{R}_p^{s/\nu} f \|_{L^p(\G)} \lesssim  \| f \|_{W^{\nu,p}(\G)} \qquad \text{for every } f\in W^{\nu,p}(\G). 
\end{equation}

Apart from that, according to \cite[Theorem 4.10]{Folland}, in the stratified case, $W^{1,p}(\G)=L^p_1(\G)$. However, this might not hold on general graded groups (see Lemma 4.18 and the remark below in \cite{FR-Sobolev}). 

Using the following Pseudo Poincaré Inequality, we will see that for $0<s<1$, $\mathbb{W}_{p;1}(\G)\subseteq L^p_s(\G)$ and, for Carnot groups, $W^{1,p}(\G)\subseteq L^p_s(\G)$.

\begin{proposition}\label{Pseudo_Poincare}(Pseudo Poincaré Inequality)
Let $\G$ be a graded Lie group, $1<p<\infty$, and $f\in \mathbb{W}_{p;1}(\G)$, then
$$\|f-T_tf\|_{L^p(\G)}\lesssim \sum_{j=1}^n t^{\sigma_j/\nu} \, \|X_jf\|_{L^p(\G)} \qquad \forall t>0.$$
In particular, if $t<1$,
$\|f-T_tf\|_{L^p(\G)}\lesssim  t^{1/\nu} \| f \|_{\mathbb{W}_{p;1}(\G)} $.
 Also, if $\G$ is Carnot and $f\in W_{1,p}(\G)$
$$\|f-T_tf\|_{L^p(\G)}\lesssim t^{1/\nu}\sum_{j=1}^{n_1}  \, \|X_jf\|_{L^p(\G)} \qquad \forall t>0$$
where the sum is over the first stratum.
\end{proposition}

\begin{proof}
Using that  the heat kernel integrates one over $\G$ and also using the Mean Value Inequality on $L^p(\G)$ (Theorem \ref{Lp_mean_value}) we have
$$ \|f-T_tf\|_{L^p(\G)}\leq\sum_{j=1}^n\|X_jf\|_{L^p(\G)}\int_{\G} |y|^{\sigma_j}|h_t(y)| \, dy.
$$
Using the estimate for the heat kernel and changing to polar coordinates we obtain,
\begin{align*}
    \|f-T_tf\|_{L^p(\G)}&\lesssim t^{-Q/\nu} \sum_{j=1}^n\|X_jf\|_{L^p(\G)}  \int_{\G} |y|^{\sigma_j} \exp\left(-c\left(\frac{|y|^\nu}{t}\right)^{\frac{1}{\nu-1}}\right) \, dy\\
    &\approx t^{-Q/\nu} \sum_{j=1}^n\|X_jf\|_{L^p(\G)}  \int_0^\infty r^{Q-1+\sigma_j} \exp\left(-c\left(\frac{r^\nu}{t}\right)^{\frac{1}{\nu-1}}\right) \, dr.
\end{align*}
By changing coordinates,
\begin{align*}
    \|f-T_tf\|_{L^p(\G)}&\lesssim  \sum_{j=1}^n t^{\sigma_j/\nu} \, \|X_jf\|_{L^p(\G)}  \int_0^\infty s^{Q-1+\sigma_j} \exp\left(-cs^{\frac{\nu}{\nu-1}}\right) \, ds\\
    &\approx \sum_{j=1}^n t^{\sigma_j/\nu} \, \|X_jf\|_{L^p(\G)}  \int_0^\infty \tau^{(Q+\sigma_j)(1-\frac{1}{\nu})} e^{-\tau} \, \frac{d\tau}{\tau}\\
    &= \sum_{j=1}^n t^{\sigma_j/\nu} \, \|X_jf\|_{L^p(\G)}  \, \Gamma((Q+\sigma_j)(1-{1}/{\nu})).
\end{align*}
If in particular $t<1$, since $\sigma_j\geq 1$ for all $j$ we have
$$\|f-T_tf\|_{L^p(\G)}\lesssim \sum_{j=1}^n t^{1/\nu} \, \|X_jf\|_{L^p(\G)} = t^{1/\nu} \| f \|_{\mathbb{W}_{p;1}(\G)}.$$
The last part follows analogously using Remark \ref{remark: saka}.
\end{proof}

\begin{proposition}\label{Prop_cont_rockland}
Let $\G$ be a graded Lie group, $1< p<\infty$.  If $0<s<1$, then  $\mathbb{W}_{p;1}(\G) \subset L^p_s(\G)$ and
$\mathcal{R}_p^{s/\nu}$ is a bounded operator from $\mathbb{W}_{p;1}(\G)$ to $L^p(\G)$.
(For Carnot groups also  $\mathbb{W}_{p;1}(\G) \subset L^p_s(\G)$ and
$\mathcal{R}_p^{s/\nu}$ is a bounded operator from $W^{1,p}(\G)$ to $L^p(\G)$.)
\end{proposition}

\begin{proof} Consider $\alpha=s/\nu$. We will apply Theorem \ref{Teo_MCSA}, that is, we will see that the integral \eqref{app_teo_A_alpha} defining $\mathcal{R}^\alpha_p$ exists as a Bochner integral for every function in $\mathbb{W}_{p;1}(\G)$. In other words, we will prove that if $f\in \mathbb{W}_{p;1}(\G)$,  the following integral is finite
\begin{align*}
 \int_{0}^\infty t^{-\alpha-1} \| f-T_t f\|_{L^p(\G)}  \; dt.
\end{align*}
We split the integral in $t<1$ and $t\geq 1$.
On the one hand, by the Pseudo Poincaré Inequality (Proposition \ref{Pseudo_Poincare})
$$\|f-T_tf\|_{L^p(\G)}\lesssim t^{1/\nu} \| f \|_{\mathbb{W}_{p;1}(\G)} \qquad \forall \,  0<t<1$$
thus,
$$ \int_{0}^1 t^{-\alpha-1} \| f-T_t f\|_{L^p(\G)}  \; dt
\leq \int_{0}^1 t^{-\alpha-1}  t^{1/\nu} \| f \|_{\mathbb{W}_{p;1}(\G)} \, dt
\lesssim \| f \|_{\mathbb{W}_{p;1}(\G) }
$$
provided that $1/\nu-\alpha>0$.  
On the other hand, since $\alpha>0$,
$$ \int_{1}^\infty t^{-\alpha-1} \| f-T_t f\|_{L^p(\G)}  \; dt
\leq \int_{1}^\infty t^{-\alpha-1} \, 2 \|f \|_{L^p(\G)} \, dt 
\lesssim \| f \|_{L^p(\G)}. $$
\end{proof}

\subsubsection{Characterization of fractional powers of a Rockland operator in terms of second order differences}

In this section we give a formula for the fractional powers of a Rockland operator which generalizes a well-known formula for the fractional Laplacian in Euclidean spaces. As an application, we derive other relations between the different kinds of Sobolev spaces defined in the introduction.

In order to do so, we define a kernel $k_\alpha$ in terms of the heat kernel $h_t$, 
\begin{equation}
   k_\alpha(y)= \frac{1}{\Gamma(-\alpha)} \int_0^\infty  h_t(y) \;
\frac{dt}{t^{1+\Re(\alpha)}}. \label{fractional-kernel} 
\end{equation}
The following lemma summarizes the key properties of this kernel.
\begin{lemma}\label{lemma-k-alpha}
Let $\alpha\in\C$ such that $\Re(\alpha)>-Q/\nu$. Then, kernel $k_\alpha$ is well-defined and satisfies the estimate
\be |k_\alpha(y)| \lesssim  |y|^{-(Q+\Re(\alpha) \nu)} \label{k-alpha-bound} \ee
and the scaling property 
\be k_\alpha(D_s(x)) = s^{-Q-\alpha \nu} k_\alpha(x). \label{k-alpha-scalling}  \ee
\end{lemma}

\begin{proof}
Using Theorem \ref{theorem-DHZ}, we have that
\begin{align*}
|k_\alpha(y)|  \leq C \; \int_0^\infty t^{-Q/\nu} \; 
\exp\left(-c |y|^{\frac{\nu}{\nu-1}}/ t^{1/(\nu-1)}   \right) \; \frac{dt}{t^{1+\Re(\alpha)}}.
\end{align*}
By making the change of variables
$ u = c |y|^{\frac{\nu}{\nu-1}}/ t^{1/(\nu-1)}$, 
we get
\begin{align*}
	|k_\alpha(y)| \lesssim   \; |y|^{-(Q+\Re(\alpha) \nu)} \int_0^\infty u^{(Q/\nu+\Re(\alpha))(\nu-1)} \; 
	e^{-u} \; \frac{du}{u} \lesssim |y|^{-(Q+\Re(\alpha) \nu)}
\end{align*}
provided that $Q+\Re(\alpha)\nu>0$. 
Hence,  \eqref{k-alpha-bound} follows. Finally, 
\eqref{k-alpha-scalling} follows easily from the scaling property of the heat kernel by making a change 
of variables in the definition of $k_\alpha$.

\end{proof}

\begin{remark}
In the Euclidean setting, considering $\alpha>0$, we have
$$ k_\alpha(x)= \frac{c(n,\alpha)}{|x|^{n+2\alpha}}. $$
\end{remark}

The following result extends the formula in \cite[Lemma 3.2]{DPV} to our setting.

\begin{theorem}\label{Teo_R_puntual}
	Let $\G$ be a graded Lie group,  and   $\alpha\in\C$ such that $0<\Re(\alpha)<\frac{2}{\nu}$.  
	For every  $f \in \mathcal{S}(\G)$ we have the integral representation
	\begin{equation*}\label{R_puntual}
    \mathcal{R}_p^{\alpha} f (x)= \int_{\G} [f(x\cdot y) + f(x\cdot y^{-1})-2f(x) ] \, k_\alpha(y) \; dy .	    
	\end{equation*}
\end{theorem}

\begin{proof}
	From the Theorem \ref{thm-fractional-powers} with $m=1$, we have that 
	$$ \mathcal{R}^\alpha f = \frac{1}{\Gamma(-\alpha)} \lim_{\varepsilon  \to 0} 
	\int_\varepsilon^\infty t^{-\alpha-1} (T_t f-f)  \; dt. $$
	Since $\G$ is unimodular, we have that for $f \in L^p(\G)$,
	$$ T_t f(x) = \int_{\G} f(y) h_t(y^{-1}\cdot x) \; dy = \int_{\G} f(x\cdot y^{-1}) h_t(y) \; dy. $$
	The heat kernel satisfies
	$ h_t(y) = h_t(y^{-1}) $
	by Theorem \ref{ht_properties}, and since $\mathcal{R}$ has real coefficients, $h_t$ is real-valued.
	Hence, 
	$$ T_t f(x) = \int_{\G} f(x\cdot y ) h_t(y) \; dy .$$
	Also, from 
	$ \int_{\G} h_t(y) \; dy = 1 $
	 we deduce that 
	$$ T_t f(x) - f(x) = \int_{\G} [f( x\cdot y)+ f( x\cdot y^{-1})-2f(x)] \, h_t(y) \; dy .$$
	On the other hand, we consider the function defined by \eqref{fractional-kernel}.
	Replacing, and using Fubini theorem, we get the desired formula. Indeed, Fubini theorem holds since
	\begin{equation*}
	\int_\G\int_{0}^\infty  |\Delta_y^2f(x)h_t(y)| \, \frac{dt}{t^{1+\alpha}} \, dy  = I_1 + I_2  
	\end{equation*}
	where 
	\begin{equation*}
	    I_1=\int_{|y|>1} |\Delta_y^2f(x)|\int_{0}^\infty   |h_t(y)|\, \frac{dt}{t^{1+\alpha}} \, dy
	\end{equation*}
	and 
	\begin{equation*}
	    I_2=\int_{|y|\leq 1} |\Delta_y^2f(x)|\int_{0}^\infty  |h_t(y)| \, \frac{dt}{t^{1+\alpha}} \, dy,
	\end{equation*}
are both finite. Indeed, by repeating the argument for the boundedness \eqref{k-alpha-bound}	and using polar coordinates on $\G$ \cite[Proposition 3.1.42]{FR-book} we have
	\begin{equation*}
	    I_1\leq 3\|f\|_{L^\infty(\G)} \int_{|y|>1} |y|^{-(Q+\Re(\alpha)\nu)} \, dy\lesssim \int_{r>1} r^{-\Re(\alpha)\nu-1} \, dr<\infty
	\end{equation*}
since $\Re(\alpha)>0$. On the other hand, using inequality 
\eqref{desig_2do_orden} for $f\in \mathcal{S}(\G)$, 
\begin{equation*}
    I_2\lesssim \sum_{j,k=1}^n\left\|X_jX_kf\right\|_{L^\infty(\G)} \int_{|y|\leq 1}|y|^2 \int_{0}^\infty  |h_t(y)| \, \frac{dt}{t^{1+\alpha}} \, dy  .  
\end{equation*}
Arguing as in Lemma \ref{lemma-k-alpha}, we obtain that
\begin{align*}
    I_2&\lesssim \int_{|y|\leq 1}  |y|^{2-(Q+\Re(\alpha) \nu)} \; dy 
    \lesssim \int_{r\leq 1}  r^{2-(Q+\Re(\alpha) \nu)+Q-1} \; dr = \int_{r\leq 1}  r^{1-\Re(\alpha) \nu} \; dr<\infty
\end{align*}
provided that $2-\Re(\alpha)\nu>0$.
\end{proof}

As a corollary, the  characterization given in  Theorem \ref{Teo_R_puntual} for the fractional powers of the Rockland operator holds for a wider class of functions by considering Bochner integrals.

\begin{corollary}\label{Coro_rock_characterization}
If $f\in \mathbb{W}_{p;\nu}(\G)$, $1\leq p<\infty$, and $0<\Re(\alpha)<\frac{2}{\nu}$, then 
$$ \mathcal{R}_p^{\alpha} f = \int_{\G}  k_\alpha(y) \, \Delta_y^{2}f  \; dy $$
in the sense of Bochner integral. 
\end{corollary}
\begin{proof}
Since $\nu\geq 2$, $f\in \mathbb{W}_{p;2}(\G)$ and so the following
$$ \int_\G k_\alpha(y)\, \Delta_y^2f \,  dy$$
is a Bochner integral.
Indeed, using Theorem \ref{teo_second order_for_Sobolev}    \begin{align*}
    \int_\G &\|k_\alpha(y) \, \Delta_y^2f\|_{L^p(\G)} \, dy\\    &= \int_{|y|\leq 1}\|\Delta_y^2f\|_{L^p(\G)}|k_\alpha(y)| \, dy + \int_{|y|> 1}\|\Delta_y^2f\|_{L^p(\G)}|k_\alpha(y)| \, dy\\
   &\leq \sum_{j,k=1}^n\left\|X_jX_kf\right\|_{L^p(\G)}\int_{|y|\leq 1}|y|^2|k_\alpha(y)| \, dy
   +3\|f\|_{L^p(\G)} \int_{|y|> 1}|k_\alpha(y)| \, dy
\end{align*}
which are finite, provided that  and  $0<\Re(\alpha)\nu<2$,  by using \eqref{k-alpha-bound} in Lemma \ref{lemma-k-alpha} and polar coordinates on $\G$ (\cite[Proposition 3.1.42]{FR-book}). 
We know, from Theorem \ref{Teo_R_puntual}, that 
$$ \mathcal{R}_p^{\alpha} f(x) = \int_{\G} k_\alpha(y) \, \Delta_y^2f(x)   \; dy \qquad \forall f\in \mathcal{S}(\G).$$
Arguing by density we will obtain the result.
Indeed, given $f\in \mathbb{W}_{p;\nu}(\G)$ we consider a sequence of functions $f_m \in \mathcal{S}(\G)$ such that
$$ f_m \to f \quad \hbox{in} \;\mathbb{W}_{p;\nu}(\G) .$$
We have that
$$ \mathcal{R}_p^{\alpha} f_m = \int_{\G} k_\alpha(y) \, \Delta_y^{2} f_m   \; dy. $$
On the one hand, from \eqref{desig_aux}  we have
$$ \mathcal{R}_p^\alpha f_m \to \mathcal{R}_p^\alpha f \quad \hbox{in} \, L^p(\G).$$
On the other hand, by applying the version of the dominated convergence theorem for the Bochner integral we have
$$
 \lim_{n\to\infty}\int_{\G} k_\alpha(y) \, \Delta_y^{2} f_m   \; dy =
 \int_{\G} k_\alpha(y) \, \Delta_y^{2} f \, dy \quad \text{ in } L^p(\G)
$$
since $$k_\alpha(y) \,\Delta_y^{2} f_m \to  k_\alpha(y) \,\Delta_y^{2} f \quad \text{ in } L^p(\G), \, \text{ for every } y\in\G$$ and $$\|k_\alpha(y)\Delta_y^{2} f_m\|_{L^p(\G)} \leq g(y), $$
for some $g\in L^1(\G)$: In fact, given $\varepsilon>0$ there exists $M\in\N$ such that for all $m>M$, $\|f_m-f\|_{\mathbb{W}_{p;\nu}(\G)}< \varepsilon$. Then,
\begin{align*}
    \| \Delta_y^{2}(f_m-f)\|_{L^p(\G)}&=
\| \Delta_y^{2}(f_m-f)\|_{L^p(\G)} \, \chi_{\{y: \, |y|\leq 1\}} \\
&\qquad + \| \Delta_y^{2}(f_m-f)\|_{L^p(\G)}\, \chi_{\{y: \, |y|> 1\}}\\
&\lesssim \left(\sum_{j,k=1}^n \|X_jX_k(f_m-f_m)\|_{L^p(\G)}\right)|y|^2 \, \chi_{\{y: \, |y|\leq 1\}}\\
& \qquad +3\|f_m-f\|_{L^p(\G)} \, \chi_{\{y: \, |y|\leq 1\}}\\
&\leq C_{\varepsilon} \left[|y|^2 \, \chi_{\{y: \, |y|\leq 1\}}+ \chi_{\{y: \, |y|> 1\}}\right] \qquad \forall n\geq M
\end{align*}
and so we can take
\begin{align*}
g(y):= k_\alpha(y)\left( \| \Delta_y^{2}f\|_{L^p(\G)} + C_{\varepsilon} \left[|y|^2 \, \chi_{\{y: \, |y|\leq 1\}}+ \chi_{\{y: \, |y|> 1\}}\right]   \right)
\end{align*}
which is integrable following the same arguments given at the beginning of this proof.
\end{proof}

\begin{remark}
Finally, we point out that by the symmetry of $k_\alpha$,
$$ \int_{|y|\geq  \varepsilon} [f(x\cdot y) + f(x\cdot y^{-1})-2f(x) ] \, k_\alpha(y) \; dy = 2 \int_{|y|\geq \varepsilon}  [f(x\cdot y) - f(x) ] \, k_\alpha(y) \; dy   $$
so we can write
\begin{align*} 
\mathcal{R}_p^\alpha f &=  \int_{\G} \left(\Delta_y^2f\right) \, k_\alpha(y) \, dy =
\lim_{\varepsilon \to 0} \int_{\G} \left(\Delta_y^2f\right) \, k_\alpha(y) \, dy \\
&=  2 \, \hbox{P.V.} \int_\G [f(y) - f(x) ] \, k_\alpha(y\cdot x^{-1}) \; dy  
\end{align*}
For $0<Re(\alpha)<\frac{1}{\nu}$ the principal value can be omitted, as the integral is convergent (see  \cite{DPV} for the situation in the Euclidean setting).
\end{remark}

\section{The Littlewood-Paley-Stein g--function}\label{sec_g-function}

In this section, we will prove Theorem \ref{g_equiv_norm}. We first consider the case $p=2$.

\subsection{The L2 case}

  In this case, the result follows from spectral theory. Indeed, the following lemma asserts that the square function   $g_\alpha$ defined by \eqref{def_g_alpha} is up to a constant an isometry in 
$L^2(\G)$.

\begin{lemma}\label{usando_spectral_th} Let $\G$ be a graded Lie group. 
For any $\alpha \in \C$ with $\Re(\alpha) > 0$, $g_\alpha(f)$ is well defined for $f\in L^2(\G)$ and we have that 
$$ \| g_\alpha(f) \|_{L^2(\G)} = c \; \| f \|_{L^2(\G)} $$
with $c=\Gamma(2\Re(\alpha))$.
In particular, $g_\alpha(f)(x)$ is finite for almost every $x \in \G$.
\end{lemma}

\begin{proof}
We use the Spectral Theorem applied to the self-adjoint unbounded operator 
$\mathcal{R}$ in the Hilbert space $L^2(\G)$  (see for instance 
\cite[Theorem VIII.6]{Reed-Simon}). Equation \eqref{frac_calculus_esp_th}  asserts that 
there exists a projection-valued measure  $P_\lambda$ such that 
$$ \mathcal{R} = \int_0^\infty \lambda \; dP_\lambda .$$
For any fixed $t \in (0,+\infty)$, we consider the function 
$$ \psi_t(\lambda) = (t\lambda)^\alpha e^{-\lambda t} , $$
which is bounded and Borel measurable. Then, $\psi_t(\mathcal{R})=  
(t\mathcal{R})^\alpha T_t f $ is well-defined and we have that 
\begin{align*}
\| \psi_t(\mathcal R) f \|_{L^2(\G)}^2 &=
\langle \psi_t(\mathcal R) f ,  \psi_t(\mathcal R) f  \rangle  
= \langle |\psi_t|^2(\mathcal R) f, f \rangle \\
&= \int_0^\infty |\psi_t|^2(\lambda) \; d\langle f, P_\lambda(f) \rangle \\
&= \int_0^\infty (t\lambda)^{2 \Re(\alpha)} e^{-2\lambda t} \; d\langle f, P_\lambda(f) \rangle.
\end{align*}
Moreover, we observe that by Fubini-Tonelli theorem we have  
	\begin{align*}
		\| g_\alpha(f) \|_{L^2(\G)}^2 &= \int_{\G} \left( \int_0^\infty |(t\mathcal{R})^\alpha T_t f|^2 \frac{dt}{t} \right) \; dx \\
		&= \int_0^\infty  \left(  \int_{\G} |(t\mathcal{R})^\alpha T_t f|^2  \; dx \right) \; \frac{dt}{t} \\
		&= \int_0^\infty  \| (t\mathcal{R})^\alpha T_t f \|_{L^2(\G)}^2 \; \frac{dt}{t} \\
		&=  \int_0^\infty  \left[ \int_0^\infty (t\lambda)^{2\Re(\alpha)} e^{-2\lambda t} \; d\langle f, P_\lambda(f) \rangle \right]
		\; \frac{dt}{t} \\
		&=  \int_0^\infty  \left[ \int_0^\infty (t\lambda)^{2 \Re(\alpha)} e^{-2\lambda t} 
		\; \frac{dt}{t}    \right] \; d\langle f, P_\lambda(f) \rangle \\
		&= c^2  \int_0^\infty  d\langle f, P_\lambda(f) \rangle = c^2 \| f \|_{L^2(\G)}^2
	\end{align*}
	with
	$$ c^2 = \int_0^\infty u^{2 \Re(\alpha)} e^{-u} 
	\; \frac{du}{u} = \Gamma(2 \Re(\alpha)) < \infty. $$  
\end{proof}

\subsection{The Littlewood-Paley-Stein g--function associated to an approximation of the identity}

Given a function $\phi$ on $\G$, we consider the approximation of the identity generated 
by $\phi$ (as defined in \cite[Section 3.1.10]{FR-book})\footnote{We use the notation $\phi_{(t)}$ for approximations of the identity in order to avoid possible confusion with the function $\phi_\alpha$ introduced in Lemma \ref{phi_properties} below.}  
$$\phi_{(t)} = \frac{1}{t^{Q}} \, \, \phi \circ D_{t^{-1}}$$   
and the associated square function  
\be g_\phi f(x)= \left( \int_0^\infty |\phi_{(t)} *f(x)|^2 \frac{dt}{t} \right)^{1/2} \label{g-phi} \ee
We will prove an analog of a theorem of A. Benedeck, A. Calderón and R. Panzone \cite[page 363]{BCP} 
to the setting of homogeneous groups, which can be deduced following the arguments due to A. Kor\'{a}nyi and S. V\'{a}gi 
\cite{Korangi-Vagi}.
It follows from a vector-valued-version of the Calderón-Zygmund singular integral theory. See also 
\cite[Chapter XII]{Torchinsky} for a nice exposition in the Euclidean setting.

\begin{proposition}\label{prop_KV}
Let $E$ be a Banach space and let $k:\G\to E$ be a locally integrable function satisfying the following
condition of Hörmander type: 
\begin{equation*}
   \int_{|g|>a|h|} |k(h^{-1}\cdot g)-k(g))| \; dg \leq b 
\end{equation*}
for some positive constants $a,b>0$.
Consider the vector-valued
convolution operator $A$ defined by  
$$Af(x)=\int_\G k(y^{-1}\cdot x)f(y) \, dy \qquad (x\in \G)$$
for every function of compact support on $L^\infty(\G)$. If
$A$ is of weak type $(r,r)$ for some $r>1$, that is, there exists a positive constant $C>0$
$$|\{g\in \G:\, \|Af(g)\|_E> \lambda\}|\leq c\left(\frac{1}{\lambda}\|Af\|_{L^r(\G,E)}\right)^r,$$
then, for all $1<p<r$, $Af\in L^p(\G,E)$ and $$\|Af\|_{L^p(\G,E)}\leq c_p\|f\|_{L^p(\G)}$$
where $c_p$ depends only on $a,b,c$ and $p$.
\end{proposition}

\begin{proof}
Is a consequence of Lemma 2.2 and Theorem 2.1 in \cite{Korangi-Vagi}, which generalizes Lemma 1 in \cite{BCP}.
\end{proof}

\begin{theorem}	\label{thm-KV}
	Let $\G$ be a homogeneous group. Let $\phi:\G \to \C$ such that for some $C,\varepsilon>0$ satisfies
	\begin{enumerate}
		\item[(i)] for every $x\in\G$,
		\be |\phi(x)| \leq C \; (1+|x|)^{-Q-\varepsilon}; \label{CBP-1} \ee 
		\item[(ii)]  for every $ h \in \G$, 
		$$ \int_{\G} |\phi(x\cdot h^{-1}) - \phi(x)| \; dx \leq C \; |h|^{\varepsilon} \quad \text{ and } \quad \int_{\G} |\phi(h^{-1}\cdot x) - \phi(x)| \; dg \leq C \; |h|^{\varepsilon}; $$
		
		\item[(iii)]  $g_\phi$ is of weak type $(r,r)$ for some $r> 1$.
	\end{enumerate}    
	Then, for each $1<p<r$, $$\|g_\phi(f)\|_{L^p(\G)}\leq c_p\|f\|_{L^p(\G)}$$
	 for every $f\in L^\infty(\G)$ of compact support. 
\end{theorem}

\begin{proof}
    Since we want to use the results of \cite{Korangi-Vagi}, we consider the convolution operator 
    \begin{equation}\label{operator_A}
    Af(x)=\int_\G k(y^{-1}\cdot x)f(y) \, dy
    \end{equation}
    which kernel $k$ is given by the vector-valued function
    \begin{gather*}
        k:\G\to L^2\left((0,\infty),\frac{dt}{t}\right)\\
        k(x)(t) :=\phi_{(t)}(x).
    \end{gather*}
Then for each $x\in \G$, the norm of $Af(x)$ in $E:=L^2\left((0,\infty),\frac{dt}{t}\right)$ is equal to 
$g_\phi(x) $, and for $p\geq 1$
\begin{equation}\label{acot_A}
     \| Af \|_{L^p(\G, E)}=\|g_\phi(f)\|_{L^p(\G)}.
\end{equation}
Since $\phi$ satisfies (i), $\phi\in L^1(\G)$ and 
it is easy to check that 
\begin{equation*}
    |k(x)|\leq \frac{C \, }{|x|^Q} \qquad \forall x\in\G.
\end{equation*}
Thus, $k$ is locally integrable.
Moreover, using (ii), $k$ satisfies the following condition of Hörmander type \begin{equation}\label{cond_ii}
   \int_{|g|>\rho|h|} |k(h^{-1}\cdot g)-k(g))| \; dg \leq c 
\end{equation}
for some constant $c>0$, and where $\rho$ is the constant such that  \eqref{triang} holds. The proof of \eqref{cond_ii} is given in \cite[Theorem 5.2]{Korangi-Vagi} and also is verified exactly as in the original paper \cite{BCP} (page 364). 
Also, by \eqref{acot_A} and hypothesis (iii), the operator $A$ is of weak type $(r,r)$ for some $r>1$. 
Therefore, the hypotheses on the operator $A$ in Proposition \ref{prop_KV} are reached, and we have that for every $1<p<r$,  $Af\in L^p(\G,E)$ and $\|Af\|_{L^p(\G,E)}\leq c_p\|f\|_{L^p(\G)}$ for every $f\in L^\infty(\G)$ of compact support.  
\end{proof}

\subsection{The Lp case}

In this section, we will prove Theorem \ref{g_equiv_norm} for $1<p<\infty$. 

\begin{lemma}\label{phi_properties}
Let $\G$ be a graded Lie group. For any $\alpha \in \C$ with $\Re(\alpha)>0$, let us consider the function
\begin{equation}\label{phi_alpha}
  \phi_\alpha=\mathcal{R}^\alpha(h_1).   
\end{equation}
Then, the following assertions hold:
\begin{enumerate}
\item[(P0)] $\phi_\alpha$ is well defined and belongs to the space $\mathcal{BC}(\G)$ of bounded continuous functions on $\G$. For any fixed $x$, $\phi_\alpha(x)$ is an analytic function of $\alpha$ in the region $\Re(\alpha)>0$ of the complex plane.	
\label{phi_properties:P0}
\item[(P1)] For any $\alpha>0$ 
we have that  
$$  g_\alpha(f) =  \nu \;  g_{\phi_\alpha}(f) $$
\label{phi_properties:P1}
\item[(P2)] Let $m \in \N$. When $0<\Re(\alpha)<m$, we have the following explicit expression for $\phi_\alpha$, in terms of the time derivative of the heat kernel
\be \phi_\alpha(x) =  (-1)^m \frac{1}{ \Gamma(m-\alpha)} \int_{1}^\infty \partial_t^m h_t(x) (t-1)^{m-1-\alpha}  \; dt. 
\label{explicit-phi} \ee
\label{phi_properties:P2}
\item[(P3)] We have the following estimate for $\phi_\alpha$
$$ |\phi_\alpha(x)| \leq C \; (1+|x|)^{-Q-\Re(\alpha) \nu} \qquad \forall x\in \G$$
\label{phi_properties:P3}
\item[(P4)] As a function of $x$, $\phi_\alpha(x)$ is $C^\infty$ and we have 
$$ |X^J \phi_\alpha(x)| \leq C \; (1+|x|)^{-Q-\Re(\alpha) \nu-|J|} \qquad \forall x\in \G$$
for any vector field $X$ and any  multi-index $J$.
\label{phi_properties:P4}
\item[(P5)] The function $\phi_\alpha$ is symmetric 
$$ \phi_\alpha(x^{-1})= \phi_\alpha(x) \qquad \forall x\in \G.$$
\label{phi_properties:P5}
\end{enumerate}	
\end{lemma}

\begin{proof}
We first prove property \hyperref[phi_properties:P0]{(P0)}. 
We consider the semigroup $T_t$	generated by $\mathcal{R}$ acting in the Banach space $\mathcal{BC}(\G)$ of bounded continuous functions on $\G$, to which $h_1$ belongs 
as a consequence of Theorem \ref{theorem-DHZ}. Indeed, since $h_1 \in \mathcal{S}(\G)$,  $h_1 \in \hbox{Dom}(\mathcal{R}^k)$ for every $k \in \N$. Hence $\phi_\alpha \in \mathcal{BC}(\G)$. It follows from 
Theorem 3.1.5 in \cite{MCSA} that the mapping $\alpha \mapsto \phi_\alpha$ is analytic in the region $\Re(\alpha)>0$, and hence the mapping 
 $\alpha \mapsto \phi_\alpha(x)$ is analytic for every $x\in \G$.
	
Next, we prove \hyperref[phi_properties:P1]{(P1)}. Let us consider the operator
	\begin{align*}
		A_t f &= (t\mathcal{R})^\alpha T_t f = t^\alpha \mathcal{R}^\alpha( f* h_t) = f * \left(t^\alpha  \mathcal{R}^\alpha(h_t)\right) 
	\end{align*}
	where we have used \eqref{scalling-fractional-powers}.
	Now, by the scaling property of the heat kernel
	$$h_t(x) = t^{-Q/\nu} \; h_1(D_{t^{-1/\nu}}(x)) .$$
	Then using \eqref{fractional-powers-scalling},
	\begin{align*}
		\mathcal{R}^\alpha(h_t)(x) &= t^{-Q/\nu} t^{-\alpha} \; \mathcal{R}^\alpha(h_1)(D_{t^{-1/\nu}}(x)) = t^{-\alpha} \left(\phi_\alpha\right)_{(t^{1/\nu})}(x).
	\end{align*}
	Hence
	$$ A_t f(x)= f *  \left(\phi_\alpha\right)_{(t^{1/\nu})} $$
	and 
	\begin{align*}
		g_\alpha(f)(x) &= \left( \int_0^\infty  |A_tf(x)|^2 \frac{dt}{t} \right)^{1/2} = \left( \int_0^\infty  |f *  \left(\phi_\alpha\right)_{(t^{1/\nu})}|^2 \frac{dt}{t} \right)^{1/2}.   
	\end{align*}
	The result follows by the change of variables $\tau= t^{1/\nu}$.
	
	In order to prove \hyperref[phi_properties:P2]{(P2)}, we proceed by induction on $m$. We first consider the case $m=1$ , so $0<\Re(\alpha)<1$. Let us recall the formula
	$$ \mathcal{R}^\alpha f = \frac{1}{K_{\alpha,1}} \lim_{\varepsilon  \to 0} 
	\int_\varepsilon^\infty t^{-\alpha-1} (I-T_t) \; f \; dt .$$
	Since $T_t h_1= h_t* h_1 = h_{1+t}$, we have that
	$$ \phi_\alpha(x)= \frac{1}{\Gamma(-\alpha)} \int_0^\infty [h_{1+r}(x) - h_1(x)] \, \frac{dr}{r^{1+\alpha}} .$$
	Using the Fundamental Theorem of Calculus, we write
	$$ h_{1+r}(x) - h_1(x) = \int_{1}^{1+r} \frac{\partial h_t}{\partial t}(x) \; dt .$$
	Then, using Fubini's theorem  we get 
	\begin{align*}
		\phi_\alpha(x) &= \frac{1}{\Gamma(-\alpha)} \int_0^\infty  \left[ \int_{1}^{1+r} \frac{\partial h_t}{\partial t}(x) \; dt \right]  \frac{dr}{r^{1+\alpha}} \\
		&= \frac{1}{\Gamma(-\alpha)} \int_{1}^\infty \frac{\partial h_t}{\partial t}(x) \left[ \int_{t-1}^{\infty}  \frac{dr}{r^{1+\alpha}} \right] \; dt \\
		&= - \frac{1}{ \Gamma(1-\alpha)} \int_{1}^\infty \frac{\partial h_t}{\partial t}(x) (t-1)^{-\alpha}  \; dt .
		\end{align*}
	This formal computation is justified as the double integral
	$$ \int_0^\infty  \int_{1}^{1+r} \left| \frac{\partial h_t}{\partial t}(x)\right| \; dt \;  \frac{dr}{r^{1+\alpha}} $$
	is finite as follows from the computation below. Thus, the case $m=1$ is established.
	
	Next, assume that  \hyperref[phi_properties:P2]{(P2)} holds for $0<\hbox{Re}(\alpha)<m$ for some $m\in\N$. We will show that it also holds when 
	$m$ is replaced by $m+1$ in the region $0<\hbox{Re}(\alpha)<m+1$. 
	Indeed, if $0<\Re(\alpha)<m$, it follows from \hyperref[phi_properties:P2]{(P2)} after integration by parts (and using the functional equation for the
	gamma function) that 
	$$ \phi_\alpha(x) =  (-1)^{m+1} \frac{1}{ \Gamma(m+1-\alpha)} \int_{1}^\infty \partial_t^{m+1} h_t(x) (t-1)^{m-\alpha}  \; dt .$$
(Note that the border terms vanish thanks to  Theorem \ref{theorem-DHZ}.)
 This formula is exactly the formula in \hyperref[phi_properties:P2]{(P2)} with $m$ replaced by $m+1$.             	
However, we observe that the integral on the right hand side converges and give an analytical function in the region $0<\Re(\alpha)<m+1$. 
Using \hyperref[phi_properties:P0]{(P0)} and the principle of analytic continuation, we deduce that the same formula holds in that region.
Thus, using the principle of mathematical induction, our claim is established for every $m \in \N$.

	In order to prove  \hyperref[phi_properties:P3]{(P3)}, we choose $m>\Re(\alpha)$ and recall that according to Theorem \ref{theorem-DHZ}
	$$ |\partial_{t}^m h_t(x)| \leq C \; t^{-1-Q/\nu-m} \; 
	\exp\left(-c |x|^{\frac{\nu}{\nu-1}}/ t^{1/(\nu-1)}   \right). $$
	Then by Lemma \ref{lemma-psi-bound} in the Appendix
	\begin{align*}
		|\phi_\alpha(x)| &\leq   \frac{C}{\Gamma(1-\alpha )} \psi_{\alpha-m,\beta,\nu,c}(r) \leq C \; (1+r)^{-\nu(\alpha+\beta-1)} 
		=  C \; (1+r)^{-Q-\nu \alpha}
	\end{align*}
	with $r=|x|$ and $\beta=1+Q/\nu+m$.
	
	In a similar way, to prove  \hyperref[phi_properties:P4]{(P4)} we observe that we can differentiate under the integral sign 
	\begin{align*}
		X^J \phi_\alpha(x) 
		&= (-1)^m \frac{1}{ \Gamma(m-\alpha)} \int_{1}^\infty X^J \partial^m_t h_t(x) (t-1)^{m-\alpha}  \; dt 
	\end{align*}
	and using Theorem \ref{theorem-DHZ} and Lemma \ref{lemma-psi-bound}, we get that 
	$$
	|X^J \phi_\alpha(x)| \leq \frac{C}{\Gamma(1-\alpha )} \psi_{\alpha-m,\widetilde{\beta},\nu,c}(r) \leq C \; (1+r)^{-\nu(\alpha+\widetilde{\beta}-1)} 
	=  C \; (1+r)^{-Q-\nu \alpha-|J|_\G}
	$$
	where $\widetilde{\beta}=1+Q/\nu+|J|_\G/\nu+m$.
	
Finally, property \hyperref[phi_properties:P5]{(P5)} follows from the symmetry of the heat kernel  (see Theorem \ref{ht_properties}).
\end{proof}

Using this lemma, we have that, for $\phi_\alpha=\mathcal{R}^\alpha(h_1)$, condition (i) in Theorem \ref{thm-KV} is exactly (P3),  and condition (ii) in  Theorem \ref{thm-KV} follows from (P4) and Theorem \ref{Lp_mean_value}. From (P1) and Lemma \ref{usando_spectral_th} we have condition (iii) with $r=2$ in Theorem \ref{thm-KV}. Therefore, Theorem \ref{g_equiv_norm} holds for  $1<p<2$.
Since the operator $A$ defined in \eqref{operator_A} is self-adjoint (by (P1) and (P5) of Lemma \ref{phi_properties}), we obtain that for every $1<p<\infty$ 
$$ \| g_\alpha(f) \|_{L^p(\G)} \leq c_p \; \| f \|_{L^p(\G)} $$
for some constant $c_p$ depending on $p$ and the bounding constants of $\phi_\alpha$.

Finally, using a well-known duality argument \cite[Chapter 6]{Journe} we get the reverse 
inequality
$$ \| f \|_{L^p(\G)} \leq C \; \| g_\alpha(f) \|_{L^p(\G)}. $$

\begin{remark}
Notice that $ \int_{\G} h_t(x) \; dx =1 $
which implies  
$ \int_{\G} \frac{\partial h_t}{\partial t}(x) = 0 $, and  from where it can be deduced $$\int_\G \phi_\alpha(x) dx=0.$$
However, $\phi_\alpha$ does not satisfy $$ \int_{a<|x|<b} \phi_\alpha(x) \; dx = 0  \qquad \forall \; 0<a<b $$ 
which is assumed as an hypothesis in \cite[Theorem 5.1]{Korangi-Vagi} (see Appendix \ref{appendix-A} for an explicit computation of $\phi_\alpha$ in the Euclidean case). Due to this fact, we needed a statement such as Theorem \ref{thm-KV}, which is not a direct consequence of the results in the article \cite{Korangi-Vagi}. 
\end{remark}

\subsection{Application to Sobolev spaces}

We will prove Theorem \ref{G_s} by using the Littlewood-Paley-Stein $g$--function.

\begin{proof} (Theorem \ref{G_s})

	We recall that $L^p_s(\G)= \hbox{Dom}\left((\mathcal{R}_p)^{s/\nu}\right)$ and by \eqref{Theorem 4.3 FR-book}
	$$ \| f \|_{L^p_s(\G)} \approx \| f \|_{L^p(\G)} + \| (\mathcal{R}_p)^{s/\nu} f \|_{L^p(\G)}. $$
	From Theorem \ref{g_equiv_norm}, 
	$$ \| (\mathcal{R}_p)^{s/\nu} f \|_{L^p(\G)} \approx \| g_\alpha((\mathcal{R}_p)^{s/\nu} f) \|_{L^p(\G)}.$$    
	If we choose $\alpha=1-s/\nu>0 $,
	\begin{align*}
		g_\alpha((\mathcal{R}_p)^{s/\nu} f )
		&=  \left( \int_0^\infty |(t\mathcal{R}_p)^\alpha T_t (\mathcal{R}_p)^{s/\nu} f|^2 \frac{dt}{t}  \right)^{1/2}  \\
		&=  \left( \int_0^\infty t^{2\alpha} |\mathcal{R}_p^{\alpha+s/\nu}  T_t f|^2 \frac{dt}{t}  \right)^{1/2}  \\
		&=  \left( \int_0^\infty t^{2\alpha} |\mathcal{R}_p  T_t f|^2 \frac{dt}{t}  \right)^{1/2}  \\
		&=  \left( \int_0^\infty t^{2\alpha} \left|\frac{\partial T_t f}{\partial t}\right|^2 \frac{dt}{t}  \right)^{1/2} 
		= G_s(f) 
	\end{align*}
	since $T_t f$ satisfies the heat equation.
\end{proof}

\begin{remark}
As said in the introduction, many similar characterizations of the potential spaces using square functions are known in different contexts. For instance, in the Euclidean case a similar characterization of the potential spaces using the Poisson semigroup is given in \cite[Theorem 5]{SegoviaWheeden}. Extensions to other operators using semigroups associated with different operators are discussed in \cite{BFRM} and \cite{BFTRMTT}.
\end{remark}

\section{Strichatz characterization by differences}\label{sec_strichatz}

In this section, we prove Theorem \ref{characterization Strichartz}. This follows by almost the same arguments as in \cite{CRTN}. Since in that article the authors deal with first order differences we only write the proof of (ii). Indeed,  the proof of (i) follows with minor modifications by using the known Theorem \ref{Teorema_MeanValue} instead of Theorem \ref{MeanValue_primero} (second-order mean value inequality). 

As a consequence of Theorem \ref{characterization Strichartz} we 
obtain Corollary \ref{indepencdencia_Rockland}. This follows directly since the generalization on graded groups of the Strichartz's functionals do not depend on Rockland operators. 
This has been proved in  \cite[Theorem 4.20]{FR-Sobolev} from a different approach.

Let $\G$ be a graded group and $\mathcal{R}_p$ be a positive self-adjoint Rockland operator acting on $L^p(\G)$ of homogeneous degree $\nu$.  
Consider $S_\s^{(2)}$ given by \eqref{S_\s^{(2)}}.
We recall 
that by \eqref{Theorem 4.3 FR-book}
	
	$$ \| f \|_{L^p_s(\G)} \approx \| f \|_{L^p(\G)} + \| (\mathcal{R}_p)^{s/\nu} f \|_{L^p(\G)}. $$
 Therefore, in order to prove Theorem \ref{characterization Strichartz} we only need to verify 
$$ \| (\mathcal{R}_p)^{\s/\nu} f \|_{L^p(\G)} \approx \| S^{(2)}_\s f \|_{L^p(\G)} \qquad \forall f\in L^p_\s(\G).$$

\begin{remark} We notice that the conditions of Strichartz's characterization are optimal. Indeed,
we will give very simple counterexamples for $\G=\R$ showing that \eqref{scrtich_estim_1} (resp. \eqref{scrtich_estim_2}) does not hold if $s>1$ (resp. $s>2$).
Let $\varphi\in C^\infty(\R)$ an even function of compact support such that $\varphi(x)=1$
 for every $x\in [-1,1]$ and  $\varphi(x)=0$ for all $x\not \in (-2,2)$. 
Consider  
$$ f_1(x) := x \varphi(x) \qquad \text{ and } \qquad f_2(x): = x^2 \varphi(x).$$
Then, for each $x \in (0,1)$, let $\delta = \min(x,1-x)$, and so if  $|y|<\delta$ we have 
$$f_1(x+y)-f_1(x) = y\quad \text{and} \quad \Delta_y^2 f_2(x) = f_2(x+y)+f_2(x-y)-2f_2(x) = 2y^2  $$
Hence, if $s \geq 1$,
\begin{align}
\left(S_s f_1(x)\right)^2 &=  \int_{0}^{\infty} \left[ \int_{|y|< 1} |f_1(x+ry) - f_1(x) | \,  dy \right]^2  \frac{dr}{r^{1+2s}} \label{S_1_contra}  \\ 
 &\gtrsim  \, \int_{0}^{\delta} r^{1-2s} \,  dr=+\infty  \notag 
\end{align}
and similarly if $s\geq 2$,
\begin{align}
\left(S_s^{(2)} f_2(x)\right) &=  \int_{0}^{\infty} \left[ \int_{|y|<1} |\Delta_y^2f_2(x)| \; dy \right]^2 \; \frac{dr}{r^{1+2s}}  
\gtrsim  \int_{0}^{\delta} r^{3-2s}  \, dr=+\infty.   \label{S_2_contra}
\end{align}
Since \eqref{S_1_contra} and \eqref{S_2_contra} hold for every $x\in (0,1)$, $S_s f_1 \not \in L^p(\R)$ for every $s\geq 1$ and  $S_s f_2 \not \in L^p(\R)$ for all $s\geq 2$. However, both $f_1$ and $f_2$  belong to every potential space $L^p_\s(\R)$ since  $f_1, f_2\in C^\infty(\R)$ and have compact support.
\end{remark}

\subsection{Estimate in one direction}

\begin{proposition}\label{direccion_facil}
Let $\G$ be a graded Lie group and $\mathcal{R}_p$ be a positive self-adjoint Rockland operator acting on $L^p(\G)$ of homogeneous degree $\nu$.  For every $1<p<\infty$, and  $s>0$
$$ \| (\mathcal{R}_p)^{\s/\nu} f \|_{L^p(\G)} \lesssim \| S^{(2)}_\s f \|_{L^p(\G)} \qquad \forall f\in L^p_\s(\G).$$
\end{proposition}

\begin{proof} 
From Theorem \ref{G_s} we know that
$$\|G_{\s}(f)\|_{L^p(\G)}\approx \|(\mathcal{R}_p)^{\s/\nu}f\|_{L^2(\G)}.$$
Therefore, only need to prove that for $x \in \G$,
$$ \int_0^\infty t^{1-\alpha}  \left| \frac{\partial T_t f}{\partial t}(x) \right|^2 \; dt \lesssim  S_\s^{(2)} f(x) $$
with $\alpha=\frac{2\s}{\nu}$.

From the properties of the heat kernel and the fact that $\G$ is unimodular, we have
\begin{equation*}
    \begin{cases}
    \int_{\G} \frac{\partial h_t}{\partial t}(-y) \; dy = 0\\
 \quad \\
   \frac{\partial T_t f}{\partial t} (x)
= \int_{\G} \frac{\partial h_t}{\partial t} (y) f(x\cdot y) \; dy \\ 
\quad\\
    \frac{\partial T_t f}{\partial t} (x)
= \int_{\G} \frac{\partial h_t}{\partial t} (y) f(x\cdot y^{-1}) \; dy
     \end{cases}
\end{equation*}
and then we can write
$$ 2 \, \frac{\partial T_t f}{\partial t} (x)
= \int_{\G} \frac{\partial h_t}{\partial t} (y) \left[ f(x\cdot y) +  f(x\cdot y^{-1})-2 f(x) \right] \; dy .$$
We also know from Theorem \ref{theorem-DHZ} that there exist constants $C,c>0$ such that for each $x\in\G$, we have the estimate 
$$ |\partial_{t} h_t(x)| \leq \frac{C}{t^{1-Q/\nu}} \; 
\exp\left(-c \left(\frac{|x|^{\nu}}{ t}\right)^{\frac{1}{\nu-1}}   \right). $$
Hence,
\begin{align*}
2\, \left| \frac{\partial T_t f}{\partial t} (x) \right|
&\leq  \frac{C}{t^{1+Q/\nu}} \int_{\G}  \exp \left( - c \left(\frac{|y|^\nu}{t}\right)^{\frac{1}{\nu-1}} \right)  \left| \Delta_y^2f(x) \right| \; dy \\
&\leq  \frac{C}{t^{1+Q/\nu}} \left( A + \sum_{k=0}^\infty B_k \right) 
\end{align*}
where 
$$
A = \int_{|y|<t^{1/\nu}}  \exp \left( - c \left(\frac{|y|^\nu}{t}\right)^{\frac{1}{\nu-1}} \right)  \left| \Delta_y^2f(x) \right| \; dy ;
$$
$$
B_k = \int_{2^k t^{1/\nu} \leq |y|< 2^{k+1} t^{1/\nu}}  \exp \left( - c \left(\frac{|y|^\nu}{t}\right)^{\frac{1}{\nu-1}} \right)  \left| \Delta_y^2f(x) \right| \; dy .
$$
Using Cauchy-Schwarz inequality,
\begin{align*}
\, \left| \frac{\partial T_t f}{\partial t} (x) \right|
\leq  \frac{C}{t^{1+Q/\nu}} \left( A^2 + \sum_{k=0}^\infty B_k^2 \right)^{1/2}.
\end{align*}
It follows that
\begin{align*}
\left| \frac{\partial T_t f}{\partial t} (x) \right|^2
\lesssim  \frac{1}{ t^{2(1+Q/\nu)}} \left( \phi(t^{1/\nu}) + \sum_{k=0}^\infty  e^{-c2^{\frac{\nu}{\nu-1} k}} \phi(2^{k+1} t^{1/\nu} ) \right)
\end{align*}
with
$$ 
\phi(r)= \left( \int_{|y|<r}   \left| \Delta_y^2f(x) \right| \; dy \right)^{2}  
$$
Hence
$$ \int_0^\infty t^{1-\alpha}  \left| \frac{\partial T_t f}{\partial t}(x) \right|^2 \; dt \lesssim
 \int_0^\infty \frac{1}{t^{1+\alpha+2Q/\nu}} \left( \phi(t^{1/\nu}) + \sum_{k=0}^\infty  e^{-c2^{\frac{\nu}{\nu-1}  k}} \phi(2^{k+1} t^{1/\nu} ) \right) \; dt.
$$
For $k=-1,0,1,2,\dots$, we consider 
\begin{align*}
  I_k&=\int_0^\infty t^{-1-\alpha-2Q/\nu}\,  \phi(2^{k+1}t^{1/\nu} ) \, dt\\  
    &=\nu\, 2^{(k+1)(\nu\alpha+2Q)}\int_0^\infty r^{-1-\nu\alpha-2Q} \, \phi(r) \, dr.
\end{align*}
As a consequence,
\begin{align*}
  \int_0^\infty t^{1-\alpha}  \left| \frac{\partial T_t f}{\partial t}(x) \right|^2 \; dt&\lesssim 
  \left[1+\sum_{k=0}^\infty  e^{-c2^{\frac{\nu}{\nu-1}  k}} \, 2^{(k+1)(2Q+\nu\alpha+\nu)}\right]\\
  &\qquad \times \int_0^\infty\left[\frac{1}{r^{\alpha\frac{\nu}{2} + Q}}\int_{|y|<r}   \left| \Delta_y^2f(x)\right| \; dy\right]^2 \, \frac{dr}{r}.
\end{align*}
Therefore, since $\alpha=\frac{2\s}{\nu}$
we obtain
$$ \int_0^\infty t^{1-\frac{2\s}{\nu}}  \left| \frac{\partial T_t f}{\partial t}(x) \right|^2 \; dt\lesssim S_\s^{(2)}f(x).$$
\end{proof}

\subsection{Estimate on the other direction}

\begin{proposition}
Let $\G$ be a graded Lie group and $\mathcal{R}_p$ be a positive self-adjoint Rockland operator acting on $L^p(\G)$ of homogeneous degree $\nu$.  For every $1<p<\infty$, 
 if $\s\in (0,2)$, the following estimate holds 
$$ \| S^{(2)}_\s f \|_{L^p(\G)} \lesssim \| (\mathcal{R}_p)^{\s/\nu} f \|_{L^p(\G)} \qquad \forall f\in L^p_\s(\G).$$
\end{proposition}

\begin{proof}
Following \cite{CRTN}, we perform a Littlewood-Paley decomposition of $f$
\begin{align*}
f &= -\int_0^\infty \frac{\partial T_t f}{\partial t} \; dt = \sum_{m=-\infty}^\infty f_m  
\end{align*}
where
\begin{align*}
f_m =  -\int_{2^m}^{2^{m+1}} \frac{\partial T_t f}{\partial t} \; dt \qquad \text{ and } \qquad  g_m =  \int_{2^{m-1}}^{2^{m}} \left|\frac{\partial T_t f}{\partial t}\right| \; dt ,
\end{align*}  
so that for every $x\in \G$ and $m \in \Z$, we have
\begin{equation}\label{eq(8)}
   |f_m(x)| \leq g_{m+1}(x). 
\end{equation}
Since $T_t$ is a semigroup $T_{2t}= T_t T_t$, therefore
\begin{align*}
f_m &= -2 \int_{2^{m-1}}^{2^{m}} \frac{\partial T_{2t} f}{\partial t} \; dt 
= -2 \int_{2^{m-1}}^{2^{m}} T_t  \frac{\partial T_{t} f}{\partial t} \; dt. 
\end{align*}  
We deduce that
\begin{align*}
f_m(y) &= -2 \int_{2^{m-1}}^{2^{m}} \int_{\G} h_t(z^{-1}\cdot y)   \frac{\partial T_{t} f(z)}{\partial t} \; dz \; dt
\end{align*}
and hence, for any multi-index $\beta\in\N_0^n$,
\begin{align*}
X^\beta f_m(y) &= -2 \int_{2^{m-1}}^{2^{m}} \int_{\G} X^\beta h_t(z^{-1}\cdot y)   \frac{\partial T_{t} f(z)}{\partial t} \; dz \; dt.
\end{align*}

Fixed $x\in\G$ and $m\in\Z$, let $y\in\G$ such that $|y^{-1}\cdot x|^\nu\leq 2^{m-1}$. Also, let $t\in[2^{m-1},2^{m}]$. Then, there exists constants $a,b>0$ such that 
\begin{equation*}\label{cuenta_desig_triang}
  e^{-\left(\frac{|z^{-1}\cdot y|^\nu}{t}\right)^{\frac{1}{\nu-1}}}\leq b \,  e^{-a\left(\frac{|z^{-1}\cdot x|^\nu}{t}\right)^{\frac{1}{\nu-1}}}.  
\end{equation*}
Indeed, by triangle inequality and since $\nu>1$
\begin{equation*}
    \frac{|z^{-1}\cdot x|^\nu}{t}\leq \frac{(|y^{-1}\cdot x|+|z^{-1}\cdot y|)^\nu}{t}\leq 2^{\nu-1} \frac{2^{m-1}
    +|z^{-1}\cdot y|^\nu}{t}\leq  2^{\nu-1}\left(1+\frac{|z^{-1}\cdot y|^\nu}{t}\right).
\end{equation*}
Since $\nu\geq 2$, if $\nu=2$, the exponent $\frac{1}{\nu-1}$ is one, and we have 
$$e^{-\frac{|z^{-1}\cdot y|^2}{t}}\leq e^{-\left(\frac{1}{2}\frac{|z^{-1}\cdot x|^2}{t} - 1 \right)}
    =e \cdot e^{-\frac{1}{2}\frac{|z^{-1}\cdot x|^2}{t}}.$$
If $\nu>2$, then $\frac{1}{\nu-1}<1$ and 
\begin{align*}
    \frac{1}{2}\left(\frac{|z^{-1}\cdot x|^\nu}{t}\right)^{\frac{1}{\nu-1}}&\leq \left(1+\frac{|z^{-1}\cdot y|^\nu}{t}\right)^{\frac{1}{\nu-1}}
    \leq 1 + \left(\frac{|z^{-1}\cdot y|^\nu}{t}\right)^{\frac{1}{\nu-1}},
\end{align*}
 so
    $$e^{-\left(\frac{|z^{-1}\cdot y|^\nu}{t}\right)^{\frac{1}{\nu-1}}}\leq e^{-\frac{1}{2}\left(\frac{|z^{-1}\cdot x|^\nu}{t}\right)^{\frac{1}{\nu-1}} + 1 }
    =e \cdot  e^{-\frac{1}{2}\left(\frac{|z^{-1}\cdot x|^\nu}{t}\right)^{\frac{1}{\nu-1}}}$$
Thus, we obtain  \eqref{cuenta_desig_triang} with  constants $b=e$ and $a=\frac{1}{2}$.

Hence, using the estimate given in Theorem \ref{theorem-DHZ} we deduce that
\begin{align*}
|X^\beta &f_m(y)| \leq 2  \int_{2^{m-1}}^{2^{m}} \int_{\G} |X^\beta h_t(z^{-1}\cdot y)|   \left|\frac{\partial T_{t} f(z)}{\partial t}\right| \; dz \; dt \\
& \lesssim  \; \int_{2^{m-1}}^{2^{m}} \int_{\G}  t^{-|\beta|_\G/\nu-Q/\nu}  \exp\left(- c \; \left(\frac{|z^{-1}\cdot y|^\nu}{t}\right)^{\frac{1}{\nu-1}}\right) \left|\frac{\partial T_{t} f(z)}{\partial t}\right| \; dz \; dt \\
& \lesssim  \frac{1}{2^{m(|\beta|_\G/\nu+Q/\nu)} } \int_{2^{m-1}}^{2^{m}} \int_{\G}  \exp\left(- \frac{c}{2} \; \left(\frac{|z^{-1}\cdot x|^\nu}{t}\right)^{\frac{1}{\nu-1}}\right)   \left|\frac{\partial T_{t} f(z)}{\partial t}\right| \; dz \; dt \\
& \lesssim  \frac{1}{2^{m(|\beta|_\G/\nu+Q/\nu)}}  \int_{\G} g_m(z)  \exp\left(- \frac{c}{2} \; \left(\frac{|z^{-1}\cdot x|^\nu}{2^m}\right)^{\frac{1}{\nu-1}}\right) \; dz.
\end{align*}
So,
\begin{align*}
|X^\beta f_m(y)| \lesssim & \,   \frac{1}{2^{m|\beta|_\G/\nu}} \left[ \frac{1}{(2^{m/\nu})^Q} \int_{|z^{-1}\cdot x|^\nu \leq 2^m} g_m(z)  \exp\left(-\frac{c}{2} \; \left(\frac{|z^{-1}\cdot x|^\nu}{2^m}\right)^{\frac{1}{\nu-1}}\right) \; dz \right. \\
& +
\left. \sum_{k=0}^\infty  \frac{1}{(2^{m/\nu})^Q}\int_{2^{k+m}<|z^{-1}\cdot x|^\nu \leq 2^{k+m+1}} g_m(z)  \exp\left(- \frac{c}{2} \; \left(\frac{|z^{-1}\cdot x|^\nu}{2^m}\right)^{\frac{1}{\nu-1}}\right) \; dz \right].
\end{align*}
We deduce that,
\begin{align}\label{acot_maximal} 
|X^\beta f_m(y)| &\lesssim 
 \frac{1}{2^{m\beta|_\G/\nu}} \left[ 1+  \sum_{k=0}^{+\infty} 2^{k Q/\nu}  \exp\left(-\frac{c}{2}\, 2^{k/(\nu-1)}\right) \right] \mathcal{M}(g_m)(x) \notag \\ &\lesssim \frac{1}{2^{m|\beta|_\G/\nu}} \; \mathcal{M}(g_m)(x) 
\end{align}
for all $y$  in the ball $B(x,2^{(m-1)/\nu})$ where $\mathcal{M}$ denotes the Hardy-Littlewood maximal function
$$\mathcal{M}(g)(x)=\sup_r \frac{1}{r^Q}\int_{|y^{-1}\cdot x|\leq r}|g(y)|dy.$$

Now we compute
\begin{align*}
\left(S^{(2)}_\s f(x)\right)^2 &=  \int_{0}^{\infty} \left[ \frac{1}{r^{v + Q}} \int_{|y|\leq r} |\Delta^2_y f(x) | \; dy \right]^2 \; \frac{dr}{r}  \\
&= \sum_{j=-\infty}^\infty  \int_{2^j}^{2^{j+1}}  \left[ \frac{1}{r^{\s + Q}} \int_{|y|\leq r} |\Delta^2_y f(x) | \; dy \right]^2 \; \frac{dr}{r} \\
&\lesssim \sum_{j=-\infty}^\infty  \left[ \frac{1}{2^{j(\s+Q)}} \int_{|y|\leq 2^{j+1}} 
	|\Delta^2_y f(x) | \; dy \right]^2 .
\end{align*}
Using the decomposition of $f$ by means of the $f_m$'s one can write 
$$
\int_{|y|\leq 2^{j+1}} 
|\Delta^2_y f(x) | \; dy
\leq \sum_{m=-\infty}^\infty \int_{|y|\leq 2^{j+1}} 
|\Delta^2_y f_m(x) | \; dy.
$$
When $j+1\leq (m-1)/\nu$, and  $|y|\leq 2^{j+1}$, by Theorem \ref{MeanValue_primero} and \eqref{acot_maximal} we have
\begin{align*}
    	|\Delta^2_y f_m(x)| &\lesssim 
    	\sum_{|\beta|\leq 2, \,  |\beta|_\G>1}|y|^{|\beta|_\G}\sup_{|x^{-1} \cdot z|\leq \eta|y| }|X^\beta f_m(z)|\\
    	&\lesssim  \sum_{|\beta|\leq 2, \,  |\beta|_\G>1}2^{(j+1)|\beta|_\G}\sup_{|x^{-1}\cdot z|\leq \eta 2^{(m-1)/\nu} }|X^\beta f_m( z)|\\
    	&\lesssim \sum_{|\beta|\leq 2, \,  |\beta|_\G>1} 2^{(j+1)|\beta|_\G} \frac{1}{2^{m|\beta|_\G/\nu}} \mathcal{M}(g_m)(x)\\
    	&= \sum_{|\beta|\leq 2, \,  |\beta|_\G>1} 2^{|\beta|_\G(j+1-m/\nu)}  \mathcal{M}(g_m)(x)
\end{align*}
Notice that for a multi-index $\beta$ satisfying $|\beta|\leq 2$ and $|\beta|_\G>1$ we have $|\beta|_\G\geq 2$ since $\G$ is graded and so the dilation weights can be assumed as positive integers. Then, since 
since $j+1-m/\nu\leq -1/\nu$ and so negative, we have
\begin{equation*}
    |\Delta^2_y f_m(x)| \lesssim 2^{2(j+1-m/\nu)}  \mathcal{M}(g_m)(x)\lesssim 2^{2(j-m/\nu)}  \mathcal{M}(g_m)(x) 
\end{equation*}
Therefore, 
$$\frac{1}{2^{jQ}} \int_{|y|\leq 2^{j+1}} 
	|\Delta^2_y f_m(x) | \; dy \lesssim {2^{2j-2m/\nu}} \mathcal{M}(g_m)(x).$$
When $j+1> (m-1)/\nu$, we use \eqref{eq(8)} to write
\begin{align*}
  \frac{1}{2^{jQ}} \int_{|y|\leq 2^{j+1}} 
	|\Delta^2_y f_m(x) | \; dy &\leq \frac{2}{ 2^{jQ}} \int_{|y|\leq 2^{j+1}} 
	|f_m(x\cdot y) | \; dy+ \frac{2}{ 2^{jQ}} \int_{|y|\leq 2^{j+1}} 
	|f_m(x) | \; dy \\
	&\lesssim  \mathcal{M}(f_m)(x) \lesssim  \mathcal{M}(g_{m+1})(x).  
\end{align*}
Finally,
\begin{equation*}
\frac{1}{ 2^{jQ}} \int_{|y|\leq 2^{j+1}} 
	|\Delta^2_y f(x) | \; dy\lesssim \sum_{m=-\infty}^{\nu(j+1)} \mathcal{M}(g_{m+1})(x) \, + \,  \sum_{m=\nu(j+1)+1}^{\infty} 2^{2j-2m/\nu}\mathcal{M}(g_{m})(x).  
\end{equation*}
As a consequence,  denoting $c_m=\mathcal{M}(g_m)(x)$ we have 
\begin{equation*}
    \frac{1}{2^{j(\s+Q)} } \int_{|y|\leq 2^{j+1}} 
	|\Delta^2_y f(x) | \; dy\lesssim 2^{-j\s} \sum_{m=-\infty}^{\nu(j+1)} c_{m+1}  \, \, + \, 2^{-j\s}\sum_{m=\nu(j+1)}^{\infty} 2^{2j-2m/\nu}c_m
\end{equation*}
and so
\begin{align*}
    S^{(2)}_\s f(x)^2&\lesssim \sum_{j=-\infty}^\infty  \left[ \frac{1}{2^{j(\s+Q)}} \int_{|y|\leq 2^{j+1}} 
	|\Delta^2_y f(x) | \; dy \right]^2\\
	&\lesssim \sum_{j=-\infty}^\infty  2^{-2j\s}\left[\sum_{m=-\infty}^{\nu(j+1)} c_{m+1}\right]^2 \, \, + \,  \sum_{j=-\infty}^\infty 2^{-2j\s}\left[\sum_{m=\nu(j+1)+1}^{\infty} 2^{2j-2m/\nu}c_m\right]^2.
\end{align*}
Let
$$A=\sum_{j=-\infty}^\infty  2^{-2j\s}\left[\sum_{m=-\infty}^{\nu(j+1)} c_{m+1}\right]^2 $$
and
\begin{align*} 
B&=\sum_{j=-\infty}^\infty 2^{-2j\s}\left[\sum_{m=\nu(j+1)+1}^{\infty} 2^{2j-2m/\nu} \, c_m\right]^2\\
&=\sum_{j=-\infty}^\infty 2^{2j(2-\s)}\left[\sum_{m=\nu(j+1)+1}^{\infty} 2^{-2m/\nu} \, c_m\right]^2.
\end{align*}
Choosing $\varepsilon\in(0,\s/\nu)$ and using Cauchy-Schwarz inequality, one obtains
\begin{align*}
    A&=\sum_{j=-\infty}^\infty  2^{-2j\s}\left[\sum_{m=-\infty}^{\nu(j+1)} 2^{m\varepsilon} \, c_{m+1} \,  2^{-m\varepsilon}\right]^2\\& \lesssim \sum_{j=-\infty}^\infty  2^{-2j\s} \left[\sum_{m=-\infty}^{\nu(j+1)} 2^{2m\varepsilon}\right] \left[\sum_{m=-\infty}^{\nu(j+1)} 2^{-2m\varepsilon} \, c_{m+1}^2\right]\\
   & \lesssim \sum_{j=-\infty}^\infty  2^{2j(\nu\varepsilon-\s)}\left[\sum_{m=-\infty}^{\nu(j+1)} 2^{-2m\varepsilon} \, c_{m+1}^2\right]\\
    &= \sum_{m=-\infty}^\infty 2^{-2m\varepsilon} \, c_{m+1}^2 \sum_{j=(n/\nu)-1}^{\infty} 2^{2j(\nu\varepsilon-\s)}\\
    &\lesssim \sum_{m=-\infty}^\infty 2^{-\frac{2m\s}{\nu}} \, c_{m}^2.
\end{align*}
Similarly, considering $\varepsilon\in(0,1-\s/2)$ -that is why we are asking $\s\in(0,2)$- and using Cauchy-Schwarz inequality, yields 
\begin{align*}
    B&=\sum_{j=-\infty}^\infty 2^{2j(2-\s)}\left[\sum_{m=\nu(j+1)+1}^{\infty} 2^{-(2m/\nu)(1-\s/2-\varepsilon)} \, c_m \, 2^{-(2m/\nu)(\s/2+\varepsilon)}\right]^2\\
    &\lesssim \sum_{j=-\infty}^\infty 2^{2j(2-\s)} \left[\sum_{m=\nu(j+1)+1}^{\infty} 2^{-\frac{4m}{\nu}(1-\s/2-\varepsilon)}\right] \left[\sum_{m=\nu(j+1)+1}^{\infty} 2^{-\frac{4m}{\nu}(\s/2+\varepsilon) \ } \, c_m^2\right]\\
    &\lesssim \sum_{j=-\infty}^\infty 2^{4j\varepsilon} \left[\sum_{m=\nu(j+1)+1}^{\infty} 2^{-\frac{4m}{\nu}(\s/2+\varepsilon)} \, c_m^2\right]\\
    &= \sum_{m=-\infty}^\infty 2^{-\frac{4m}{\nu}(\s/2+\varepsilon)} \, c_m^2 \left[ \sum_{j=-\infty}^{\frac{m-1}{\nu}-1}2^{4j\varepsilon}\right]\\
    &\lesssim\sum_{m=-\infty}^\infty 2^{-\frac{2m}{\nu}\s} \, c_m^2.
\end{align*}
Consequently,
\begin{equation*}
    S_\s^{(2)}f(x)^2\lesssim \sum_{m=-\infty}^\infty 2^{-\frac{2m}{\nu}\s}(\mathcal{M}(g_m)(x))^2
\end{equation*}
 and the proof ends after applying Fefferman-Stein theorem (which holds for spaces of homogeneous type, in particular homogeneous Lie groups) 
 \begin{align*}
     \|S_\s^{(2)}f(x)\|_{L^p(\G)}&\lesssim \left\|\left[\sum_{m=-\infty}^\infty 2^{-\frac{2m}{\nu}\s} \, (\mathcal{M}g_m)^2\right]^{1/2}\right\|_{L^p(\G)}\\
     &\lesssim \left\|\left[\sum_{m=-\infty}^\infty 2^{-\frac{2m}{\nu}\s} \, g_m^2\right]^{1/2}\right\|_{L^p(\G)}.
 \end{align*} 
 Using again Cauchy-Schwarz inequality,
 $$g_m(x)^2\leq 2^{m}\int_{2^{m-1}}^{2^m}\left|\frac{\partial T_t f}{\partial t}(x)\right|^2 \, dt$$
 we are allowed to get
 \begin{align*}
     \sum_{m=-\infty}^\infty 2^{-\frac{2m}{\nu}\s} \, g_m(x)^2
    &\lesssim  \sum_{m=-\infty}^\infty 2^{(-\frac{2m}{\nu}\s+1)} \int_{2^{m-1}}^{2^m}\left|\frac{\partial T_t f}{\partial t}(x)\right|^2 \, dt\\
     &\lesssim  \sum_{m=-\infty}^\infty 2^{(m-1)(1-\frac{2\s}{\nu})} \int_{2^{m-1}}^{2^m}\left|\frac{\partial T_t f}{\partial t}(x)\right|^2 \, dt
     \\     &\lesssim \int_0^\infty t^{1-\frac{2\s}{\nu}}  \left| \frac{\partial T_t f}{\partial t}(x) \right|^2 \; dt 
 \end{align*}
 and therefore
 \begin{align*}
     \|S_\s^{(2)}f\|_{L^p(\G)}&\lesssim \|(\mathcal{R}_p)^{\s/\nu} f\|_{L^p(\G)}.
 \end{align*} 
\end{proof}

\appendix\label{app}
\section{Appendix: Explicit formula for $\phi_\alpha$ in Euclidean spaces}

\label{appendix-A}

\subsection{Some useful special functions}

In our computations, the following special function $\psi_{\alpha,\beta,\nu,c}:[0,\infty) \to [0,\infty)$ plays 
a key role,
\begin{align}
	\psi_{\alpha,\beta,\nu,c}(r) &=  \int_{1}^\infty   t^{-\beta} \; 
	\exp\left(-c r^{\frac{\nu}{\nu-1}}/ t^{1/(\nu-1)}   \right)  \; (t-1)^{-\alpha}  \; dt \label{definition-psi}
\end{align}
where we assume $0<\alpha<1$ and $\alpha+\beta-1>0$.
By making a change of variables, we can write this integral as 
\begin{align}
	\psi_{\alpha,\beta,\nu,c}(r) 
	&=  \int_{1}^\infty   t^{-\alpha-\beta} \; \exp\left(-c r^{\frac{\nu}{\nu-1}}/ t^{1/(\nu-1)}   \right)  \; \left(1-\frac{1}{t}\right)^{-\alpha}  \; dt \notag \\ 
	&= \int_{0}^1   u^{\alpha+\beta} \; \exp\left(-c r^{\frac{\nu}{\nu-1}} u^{1/(\nu-1)}   \right)  \; \left(1-u\right)^{-\alpha}  \; \frac{1}{u^2} \; du \notag \\
	&= \int_{0}^1   u^{\alpha+\beta-2} \; \exp\left(-c r^{\frac{\nu}{\nu-1}} u^{1/(\nu-1)}   \right)  \;  \left(1-u\right)^{-\alpha}  \; du. 
	\label{definition-psi2}
\end{align}

\begin{lemma}
	We have that
	$$ \psi_{\alpha,\beta,\nu,c}(r)    \lesssim \; (1+r)^{-\nu(\alpha+\beta-1)}. $$
	\label{lemma-psi-bound}
\end{lemma}

\begin{proof}
	We split the last integral into two parts:  $\psi^0_{\alpha,\beta,\nu,c}(r)$ corresponding to $0 \leq u \leq 1$
	and $\psi^1_{\alpha,\beta,\nu,c}(r)$ for $1/2 \leq u \leq 1$.
	
	When $u \geq 1/2$ we use the fact that a negative exponential is decreasing,
	\begin{align*}
		\psi^1_{\alpha,\beta,\nu,c}(r) &= 
		\int_{1/2}^{1}   u^{\alpha+\beta-2} \; 
		\exp\left(-c r^{\frac{\nu}{\nu-1}} u^{1/(\nu-1)}   \right)  \; \left(1-u\right)^{-\alpha}  \; du \\
		&\leq    \exp\left(-c r^{\frac{\nu}{\nu-1}} (1/2)^{1/(\nu-1)}  \right)  \int_{1/2}^{1}   u^{\alpha+\beta-2} \;  \left(1-u\right)^{-\alpha}  \; du \\
		&\lesssim \; \exp\left(- \tilde{c} \cdot r^{\frac{\nu}{\nu-1}}  \right).
	\end{align*}
	The last integral is finite since $\alpha<1$.
	
	When $u \leq 1/2$, then 
	$(1-u)^{-\alpha}\leq 2^\alpha$, and so
	\begin{align*}
		\psi^0_{\alpha,\beta,\nu,c}(r) &= 
		\int_{0}^{1/2}   u^{\alpha+\beta-2} \; 
		\exp\left(-c r^{\frac{\nu}{\nu-1}} u^{1/(\nu-1)}   \right)  \; \left(1-u\right)^{-\alpha}  \; du \\
		&\leq  2^\alpha 
		\int_{0}^{1/2}   u^{\alpha+\beta-1} \; 
		\exp\left(-c r^{\frac{\nu}{\nu-1}} u^{1/(\nu-1)}   \right)   \; \frac{du}{u}.  
	\end{align*}
	By the change of variables $v= c r^{\frac{\nu}{\nu-1}} u^{1/(\nu-1)}$, and since $\alpha+\beta-1>0$, we have 
	\begin{align*}
		\psi^0_{\alpha,\beta,\nu,c}(r) &\lesssim \; \int_0^{c r^{\frac{\nu}{\nu-1}} (1/2)^{1/(\nu-1)}}
		\left( \frac{v}{c r^{\frac{\nu}{\nu-1}}} \right)^{(\alpha+\beta-1)(\nu-1)} \exp(-v)  \; \frac{dv}{v} \\
		&\lesssim \; r^{-\nu(\alpha+\beta-1)} 
		\; \int_0^{\infty}
		v^{(\alpha+\beta-1)(\nu-1)-1} \exp(-v)  \; dv \\
		&\lesssim \; \Gamma((\alpha+\beta-1)(\nu-1)) \, r^{-\nu(\alpha+\beta-1)}. 
	\end{align*}
\end{proof}

When $\nu=2$, $\psi_{\alpha,\beta,\nu,c}$ can be expressed in terms of the confluent hypergeometric function introduced by Kummer in 1837 \cite{Kummer}. 
We will work with the regularized version, which is  
defined by the power series
$$ \varphi(a,b;z) = \sum_{k=0}^\infty \frac{(a)_k}{\Gamma(b+k)} \frac{z^k}{k!} .$$
Here $(a)_k$ denotes the Pochhammer symbol
$$ (a)_0=1, \, \, (a)_k = \frac{\Gamma(a+k)}{\Gamma(k)} = a(a+1)(a+2)\ldots (a+k-1) $$
(see \cite[equations 9.9.1 and 9.9.2]{Lebedev}).  Then, $\varphi(a,b;z)$ is an entire function of $a$ and $b$, for any fixed $z$.
Moreover, this function admits the integral representation 
$$ \varphi(a,b;z) = \frac{1}{\Gamma(a)\Gamma(b-a)}\int_0^1 e^{zu} u^{a-1}(1-u)^{b-a-1}\,du \qquad (\Re(b)>\Re(a)>0)$$
(see \cite[equations 9.9.2 and 9.11.1]{Lebedev}). 

Comparing with \eqref{definition-psi2}, we see that
\be \psi_{\alpha,\beta,2,c}(r) = {\Gamma(\alpha+\beta-1)\Gamma(1-\alpha)} \; \varphi ( \alpha+\beta-1, \beta; {-}cr^2). \label{psi-2} \ee 

We have established this formula in the range $0<\Re(\alpha)<1$, but by analytic continuation 
it follows that it holds for any $\alpha$ with $\Re(\alpha)>0$.

\section{The function $\phi_\alpha$ in Euclidean spaces}

We will give an explicit expression for the function $\phi_\alpha$ introduced in Lemma \ref{phi_properties}, in  the Euclidean space $\R^n$ (and with $\mathcal{R}$ being minus the usual Laplacian operator). In this case, $\nu=2$ and we have the well-known explicit 
formula for the heat kernel
$$ h_t(x) =\frac{1}{(4\pi t)^{n/2}} e^{-|x|^2/4t} = \frac{1}{(4\pi t)^{n/2}} e^{-r^2/4t}  $$ 
where $r=|x|$.

We consider $0<\Re(\alpha)<1$. Then, replacing in \eqref{explicit-phi} (with $m=1$), we get
$$ \phi_\alpha(x) = c(n,\alpha) \left[ - \frac{n}{2} \psi_{\alpha,n/2+1,2,1/4}(r) + \frac{r^2}{4}   \psi_{\alpha,n/2+2,2,1/4}(r)  \right] $$
where
$$ c(n,\alpha) = - \frac{1}{\Gamma(1-\alpha)} (4\pi)^{-n/2}  $$
and $\psi_{\alpha,\beta,\nu,c}$ is the function introduced in \eqref{definition-psi}. Using \eqref{psi-2} we obtain that
\begin{align*} 
	\phi_\alpha(x) &= 
	(4\pi)^{-n/2} 
	\left[  \frac{n}{2}  \, {\Gamma(\alpha+n/2)} \, \varphi(\alpha+n/2,n/2+1;-r^2/4) \right. \\
	& \qquad\left. - {\Gamma(\alpha+1+n/2)} \, \frac{r^2}{4} \, \varphi(\alpha+n/2+1,n/2+2;-r^2/4) \right].
\end{align*}

The confluent hypergeometric function $\varphi(a,b;z)$ can be computed with high precision using the C-library Arb \cite{Johansson,Johansson2017arb}. 
For instance, Figure \ref{figure-phi} shows a graph of $\phi_\alpha(r)$ for some particular values of the parameters, made by using a C program using this library and 
the graphics package Matplotlib \cite{Hunter:2007}. This picture exhibits some of the prominent features of $\phi_\alpha$ in agreement with Lemma \ref{phi_properties}.

\begin{figure}[H]
	\includegraphics[width=0.7\textwidth]{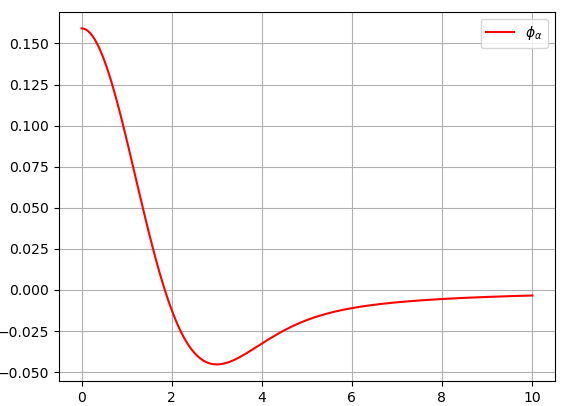}
	\caption{$\phi_\alpha(r)$ with $n=1$ and $\alpha=0.5$.}
	\label{figure-phi}
\end{figure}

\printbibliography
\end{document}